\documentclass[10pt]{amsart}
\usepackage{amsmath}
\usepackage{amssymb,amscd}
\usepackage{graphicx}
\usepackage{mathdots}
\usepackage{mathrsfs}       
\usepackage{graphicx}
\usepackage{bbm} 
\usepackage{microtype} 

\usepackage{colonequals} 
\usepackage{mathtools}

\usepackage{tikz,tkz-euclide}
\usepackage{tikz-cd}
\usetikzlibrary{arrows}
\usepackage{tikz-3dplot}

\usepackage{color}\definecolor{darkblue}{rgb}{0,0.1,.5}
\usepackage[colorlinks=true,linkcolor=darkblue, urlcolor=darkblue, citecolor=darkblue]{hyperref}
\usepackage{cleveref}

\theoremstyle{plain}
\newtheorem{theorem}{Theorem}[section]
\newtheorem{lemma}[theorem]{Lemma}
\newtheorem{proposition}[theorem]{Proposition}

\newtheorem{problem}[theorem]{Problem}

\theoremstyle{definition}
\newtheorem{definition}[theorem]{Definition}
\newtheorem{example}[theorem]{Example}

\theoremstyle{remark}
\newtheorem*{remark}{Remark}

\numberwithin{equation}{section}

\def \begineq{\begin{equation}}
\def \endeq{\end{equation}}

\def\Bier{\mathrm{Bier}}
\def\Facet{\mathrm{Facet}}

\def\pol{\mathrm{pol}}
\def\vc{\mathrm{vc}}

\def\sk{\mathrm{sk}}

\def\zk{\mathcal Z_K}
\def\rk{\mathcal R_K}

\def\Z{\mathbb Z}

\def\R{\mathbb R}

\def\N{\mathbb N}

\DeclareMathAlphabet{\mathbbmsl}{U}{bbm}{m}{sl}

\DeclareMathOperator{\cone}{Cone}

\DeclareMathAlphabet{\mathpzc}{OT1}{pzc}{m}{it}

\newcommand{\gen}[3]{{\mathpzc{#1}}_{#2}^{#3}}

\makeatletter
\@namedef{subjclassname@2020}{%
  \textup{2020} Mathematics Subject Classification}
\makeatother

\title[On the combinatorics of Murai spheres and its applications]{On the combinatorics of Murai spheres and its applications}

\author[Limonchenko]{Ivan Limonchenko}
\address{Mathematical Institute of the Serbian Academy of Sciences
and Arts (SASA), Belgrade, Serbia}
\email{ivan.limoncenko@turing.mi.sanu.ac.rs}

\author[Vavpeti\v{c}]{Ale\v{s} Vavpeti\v{c}}
\address{University of Ljubljana, Slovenia}
\email{ales.vavpetic@fmf.uni-lj.si}



\begin{document}

\maketitle

\begin{abstract}
We classify the combinatorial types of Murai spheres in dimensions $1$ and $2$, thereby showing that the corresponding convex simple polytopes have Delzant realizations. Then we describe all chordal Murai spheres $\Bier_c(M)$ with $c\in\N^m$ and $m\leq 2$. Finally, we find all possible values for the Buchstaber and chromatic numbers of arbitrary Murai spheres.
\end{abstract}

\section{Introduction}

This paper is devoted to studying combinatorial properties of a class of simplicial spheres introduced in~\cite{Mu} as generalized Bier spheres. We call them Murai spheres. Recall that for an abstract simplicial complex $K$ on $[m]=\{1,2,\ldots,m\}$ different from the whole simplex $\Delta_{[m]}$ with $m$ vertices, its Bier sphere $\Bier(K)$ was defined in~\cite{Bier} as a deleted join of the complex $K$ and its Alexander dual complex $K^\vee$. Furthermore, it was shown that $\Bier(K)$ is an $(m-2)$-dimensional PL-sphere with the number of vertices varying between $m$ and $2m$, see~\cite{Matousek, Longueville}. This construction has been studied intensively and found numerous applications in the framework of topological combinatorics~\cite{Matousek}, geometrical combinatorics~\cite{BPSZ05}, polytope theory~\cite{Zivaljevic19,Zivaljevic21}, game theory~\cite{Zivaljevic23}, combinatorial commutative algebra~\cite{HK,LZ}, toric geometry and topology~\cite{LS,LV,LZ,CYY}.

Similarly, given a $c$-multicomplex $M$ with $c\in\N^m, m\geq 1$ different from the whole polynomial algebra $S=\Bbbk[x_1,x_2,\ldots,x_m]$, its generalized Bier sphere $\Bier_c(M)$ was defined in~\cite{Mu} as a certain simplicial complex on $[|c|+m]$, $|c|=c_1+c_2+\ldots+c_m$. In the special case $c=(1,1,\ldots,1)$, the complex $\Bier_c(M)$ is isomorphic to $\Bier(K_M)$, where $K_M$ is the simplicial complex corresponding to $M$. Moreover, $\Bier_c(M)$ is a $(|c|-2)$-dimensional simplicial sphere whose number of vertices ranges between $|c|$ and $|c|+m$, see~\cite{Mu}.

It was also shown in~\cite{Mu} that Murai spheres share certain important combinatorial properties with classical Bier spheres, such as being shellable and edge decomposable simplicial spheres~\cite{BPSZ05, Mu}. Furthermore, it was proved in~\cite{LZ} that the class of flag Bier spheres coincides with the class of flag Murai spheres. Moreover, as in the case of Bier spheres, the classical Steinitz problem---namely, whether a given simplicial sphere admits a polytopal realization---remains open for Murai spheres. In fact, there exist infinitely many non-polytopal Bier spheres~\cite{Matousek}, which immediately implies the existence of infinitely many non-polytopal Murai spheres as well. However, no explicit examples of non-polytopal Murai spheres have been constructed so far. One of our aims is to further analyze the combinatorics of Murai spheres and identify the widest possible class of polytopal Murai spheres, thereby obtaining constraints for a Murai sphere to be non-polytopal.

Another aim of this paper is to study certain properties and invariants of Murai spheres that are important from the viewpoint of toric geometry (Delzant realization) and toric topology (chordality, Buchstaber and chromatic numbers). 

A lattice simple polytope is said to be a Delzant polytope if its normal fan is regular; that is, the primitive normal vectors of facets meeting at each vertex form a lattice basis. We say that a (convex) simple polytope has a Delzant realization if it is combinatorially equivalent to a Delzant polytope. In dimension $2$, every simple polytope (that is, every polygon) admits a Delzant realization. In contrast, this fails in dimension~$3$: the dodecahedron provides a counterexample. Indeed, since it has no triangular or quadrilateral facets, it admits no Delzant realization by~\cite{Del05}.

On the other hand, all simple polytopes corresponding to polytopal Bier spheres have Delzant realizations~\cite{LTZ}.
Delzant polytopes are important in toric geometry~\cite{CLS}, since they give rise to nonsingular complete toric varieties (toric manifolds), see~\cite{Danilov}. The latter complex algebraic varieties, alongside with their topological counterparts called quasitoric manifolds~\cite{DJ} that are studied in the framework of toric topology~\cite{BP02,TT}, have found various applications in bordism theory~\cite{BPR}, hyperbolic geometry~\cite{BEMPP}, topological combinatorics~\cite{BZ}, differential topology~\cite{BG}, and other areas. In the case of Bier spheres, a canonical construction of such a toric manifold was given in~\cite{LTZ}. In this paper, we classify all combinatorial Murai spheres in dimensions~$1$ and~$2$, which immediately implies that the corresponding simple polytopes admit Delzant realizations. On the other hand, in every even dimension $d\geq 6$ we construct Murai spheres that do not admit regular realizations. 

The classification of chordal Bier spheres was given in~\cite{LV}. Recall that a simplicial complex is called chordal if its $1$-skeleton contains no induced cycles of length greater than three. It was shown in~\cite{LV} that in all dimensions greater than one the class of chordal Bier spheres coincides with the class of stacked Bier spheres. Here, stacked means that the sphere can be obtained from the boundary of a simplex by consecutively applying stellar subdivisions in facets. In general, any stacked sphere is chordal by~\cite{Kalai}, whereas the converse does not hold. In this paper, we classify all the $c$-multicomplexes $M$ with $c\in\N^m, m\leq 2$ such that $\Bier_c(M)$ is a chordal sphere; examples show some of them are not stacked spheres. In fact, we prove that join of boundaries of any two simplices and boundaries of certain cyclic polytopes can arise as Murai spheres.

The (complex) Buchstaber number of a simplicial complex $K$ with $m$ vertices equals the maximal rank of a toric subgroup in the compact torus $\mathbb T^m$ acting freely on the moment-angle-complex $\zk$. When $K$ is a polytopal sphere, the existence of a quasitoric manifold over $K$ is equivalent to the maximality of the Buchstaber number $s(K)$. It was shown in~\cite{LS, LV} that for any Bier sphere this invariant attains its maximal possible value, namely $s(\Bier(K))=M-N$, where $M=f_0(\Bier(K))$ is the number of vertices of $\Bier(K)$ and $N=\dim(\Bier(K))+1$. This fact also follows from the construction of a canonical regular realization of a Bier sphere given in~\cite{LTZ}. In this paper, we prove that for any Murai sphere its Buchstaber number is either maximally possible or one less. Moreover, we provide examples showing that both possibilities can occur.

\section{Basic definitions, examples and constructions}

In this section, we collect the main definitions and recall the basic results on Murai spheres needed in the sequel. Murai spheres were introduced in~\cite{Mu} as a generalization of the Bier sphere construction from simplicial complexes to arbitrary multicomplexes. We also introduce the first examples of Murai spheres that cannot be realized as Bier spheres.

Following~\cite{Mu}, we fix the following notation throughout this section:

\begin{itemize}
\item integer number $m\geq 1$; 
\item two integer vectors, $c=(c_1,\ldots,c_m)$ and $\bar{c}=(c_{1}+1,\ldots,c_{m}+1)$, in $\N^m$;  
\item integer number $|c|:=c_1+c_2+\ldots+c_m$;
\item polynomial algebra $S=\Bbbk[x_1,\ldots,x_m]$ over a field $\Bbbk$.
\end{itemize}

To define a Murai sphere, we first need to introduce the notion of a $c$-multicomplex and the concept of multicomplex-theoretic Alexander duality. 

\begin{definition}
By a $c$-\emph{multicomplex} we mean a nonempty set $M$ of $c$-\emph{monomials} in $S$, that is: 
$$
x^{a}:=x_1^{a_1}\cdot\ldots\cdot x_{m}^{a_m}\text{ with } a_i\leq c_i\text{ for all }1\leq i\leq m
$$
such that if $m_1$ divides $m_2$ and $m_2\in M$, then $m_1\in M$.
Its {\emph{Alexander dual}} with respect to $c$ is a $c$-multicomplex $M^\vee$ defined by
$$
M^\vee := \{(x^{a})^c\,|\,x^{a}\text{ is a }c\text{-monomial such that }x^{a}\notin M\},
$$
where
$$
(x^{a})^c:=x_1^{c_1-a_1}\cdot\ldots\cdot x_m^{c_m-a_m}.
$$
We call $M$ a {\emph{full}} $c$-multicomplex if it is the set of all $c$-monomials and a {\emph{proper}} $c$-multicomplex, otherwise. 
\end{definition}

\begin{remark}
Note that for $c=(1,\ldots,1)\in\N^m$, a $c$-multicomplex $M$ is essentially the same as an \emph{(abstract) simplicial complex} on $[m]$, where $[m]:=\{1,2,\ldots,m\}$; the latter complex will be denoted by $K_M$. In particular, $M$ is $(1,\ldots,1)$-full if and only if $K_M=\Delta_{[m]}$, the power set of $[m]$.
\end{remark}

Elements of a multicomplex $M$ that are maximal with respect to the divisibility partial order relation on $M$ are called \emph{generators} of $M$; their set is denoted by $\max(M)$. Similarly, we define \emph{minimal non-elements} of $M$; their set is denoted by $\min(M)$. The same notation is used in the case of a simplicial complex, since it can be considered as a particular case of a multicomplex, see the above remark. Finally, we use angle brackets to indicate the fact that a set of $c$-monomials generates a $c$-multicomplex; in particular, we write:
$$
M=\langle m\,|\,m\in\max(M)\rangle.
$$

To describe a $c$-multicomplex $M$, one needs to know either the set $\max(M)$ of its generators, or the set of its minimal non-elements $\min(M)$. Moreover, the duality between these two objects can be easily described in the language of monomial ideals. The ideal-theoretic Alexander duality goes as follows. 

\begin{definition}
Let $I\subseteq S$ be a $c$-\emph{ideal}, that is, $I$ is generated by $c$-monomials. Its {\emph{Alexander dual}} with respect to $c$ is a $c$-ideal $I^\vee\subset S$ defined by
$$
I^\vee = \{(x^{a})^c\,|\,x^{a}\text{ is a }c\text{-monomial in }S\text{ such that }x^{a}\notin I\}.
$$
\end{definition}

Denote by $I_c(M)\subset S$ the $c$-ideal generated by all $c$-monomials that do not belong to $M$. Note that for $c=(1,\ldots,1)$, this yields the \emph{Stanley--Reisner ideal} (or, the \emph{face ideal}) of the simplicial complex $K_M$ corresponding to the $c$-multicomplex $M$. Moreover, it is easy to see that
$$
(I_c(M))^\vee = I_c(M^\vee).
$$

Given a monomial ideal $I\subset S$, denote by $G(I)$ the unique minimal set of monomial generators of $I$.
The next observation was proved in~\cite[Lemma 4.3]{LZ}.

\begin{proposition}[\cite{LZ}]\label{MinimalNonMonmomialsProp}
For any proper $c$-multicomplex $M$, one has:
$$
G(I_c(M))=\{(x^a)^c\,|\,x^a\in\max(M^\vee)\}=\min(M)\text{ and }
G(I_c(M)^\vee)=\{(x^a)^c\,|\,x^a\in\max(M)\}=\min(M^\vee).
$$
\end{proposition}

Now we are ready to describe the set on which we are going to define a Murai sphere. First, we define the sets
$$
\tilde{X}_i := \{\gen{x}{i}{0},\ldots,\gen{x}{i}{c_i}\}\text{ for each }1\leq i\leq m.
$$
Then we consider their union $\tilde{X} := \tilde{X}_1\cup\tilde{X}_2\cup\ldots\cup\tilde{X}_m$. Finally, for every $c$-monomial $x^a$, set
$$
F_c(x^a) := \tilde{X}\setminus \{\gen{x}{1}{a_1},\ldots,\gen{x}{m}{a_m}\}
$$
and, for every pure power of a variable $x_i^j$, $0<j\le c_i$, set 
$$
G(x^a;x_i^j) := F_{c}(x^a)\setminus\{\gen{x}{i}{j}\}.
$$

In~\cite{Mu}, the vertex $\gen x i j$ was denoted by $x_i^{(j)}$; we use a different notation mainly to keep the expressions for faces of a Murai sphere shorter. We use~\cite[Proposition 1.10]{Mu} as a definition of Murai sphere. 

\begin{definition}
Let $M$ be a proper $c$-multicomplex. The \emph{Murai sphere} (or, a \emph{generalized Bier sphere}) of $M$ is a simplicial complex $\Bier_c(M)$ on $\tilde{X}$ with the set of maximal faces (or, facets)
$$
\Facet_c(M)=\{G(x^a;x_i^j)\mid x^a\in M,x^a\diamond x_i^j\not\in M, a_i<j\le c_i\},
$$
where the $\diamond$ operation is defined as follows:
$$
x^a\diamond x_i^j:=x_1^{a_1}\cdots x_{i-1}^{a_{i-1}}x_i^j x_{i+1}^{a_{i+1}}\cdots x_m^{a_m}.
$$
\end{definition}

\begin{remark}
For any proper $c$-multicomplex $M$, the following crucial results were obtained in~\cite{Mu}:
\begin{itemize}
\item $\Bier_c(M)\cong\Bier_c(M^\vee)$ is a simplicial $(|c|-2)$-sphere, see~\cite[Proposition 1.10, Corollary 3.9]{Mu};
\item for $c=(1,\ldots,1)$, one has an isomorphism $\Bier_c(M)\cong\Bier(K_M)$, see~\cite[Theorem 1.13]{Mu};
\item $\Bier_c(M)$ is shellable (\cite[Theorem 2.1]{Mu}) and edge decomposable (\cite[Theorem 4.6]{Mu}).  
\end{itemize}
\end{remark}

Let us describe the set $\min(\Bier_c(M))$ of minimal non-faces of a Murai sphere $\Bier_c(M)$ in terms of its Stanley--Reisner ideal.
First, we define two polarizations for a $\bar{c}$-ideal in $S$; they will both belong to the polynomial algebra $\Bbbk[X]$ on the set of variables 
$$
X := X_1\cup X_2\cup\ldots\cup X_m,\text{ where }X_i := \{x_{i,0},\ldots,x_{i,c_i}\}\text{ for }1\leq i\leq m.
$$

\begin{definition}
The \emph{polarization} of a $\bar{c}$-ideal $I\subset S$ is the monomial ideal $\pol(I)$ in $\Bbbk[X]$ 
such that
$$
\pol(I):=(\pol(x^a)\,|\,x^a\in G(I))\subset\Bbbk[X], \text{ where }
\pol(x^a):=\prod\limits_{a_i\neq 0}x_{i,0}\ldots x_{i,a_{i}-1}.
$$

The $\ast$-\emph{polarization} of a $\bar{c}$-ideal $I\subset S$ is the monomial ideal $\pol^\ast(I)$ in $\Bbbk[X]$ such that
$$
\pol^\ast(I):=(\pol^\ast(x^a)\,|\,x^a\in G(I))\subset\Bbbk[X], \text{ where }\pol^\ast(x^a):=\prod\limits_{a_i\neq 0}x_{i,c_i}\ldots x_{i,c_i-a_{i}+1}.
$$
\end{definition}

The next statement, see~\cite[Theorem 3.6]{Mu}, describes the face ideal of $\Bier_c(M)$ and therefore the set $\min(\Bier_c(M))$ of its minimal non-faces.

\begin{theorem}\label{FaceIdealMuraiThm}
Let $M$ be a proper $c$-multicomplex. Then the Stanley--Reisner ideal of the Murai sphere $\Bier_c(M)$ is equal to the sum of the three ideals:
$$
I_{SR}(\Bier_c(M))=\pol(I_c(M))+\pol^\ast(I_c(M^\vee))+\pol(x_1^{c_{1}+1},\ldots,x_{m}^{c_{m}+1}).
$$
\end{theorem}

Here is a couple of examples showing that a generalized Bier sphere of a $c$-multicomplex $M$ depends on both the integer vector $c$ and the multicomplex $M$.

\begin{example}
Let $c=(1,1,1)$ and
$M = \{1,x,y,z,xy,xz\}$.
Then $M^\vee = \{1,x\}$
and $\Bier_c(M) = \Bier(K_M) = \partial P_4$, the boundary of a 4-gon. Let $c=(2,1,1)$ and
$M = \{1,x,y,z,xy,xz\}$.
Then $M^\vee = \{1,x,y,z,x^2,yz\}$
and $\Bier_c(M) = \Bier(K_M\sqcup \varnothing)$, that is a flag Bier sphere over $K_M$ with one ghost vertex added.
\end{example}

\begin{example}
Let $c=(2,1)$ and
$M = \{1, x, x^2, y\}$.
Then $M^\vee = \{1, x\}$
and $\Bier_c(M) = \Bier(\ast \sqcup \varnothing \sqcup \varnothing)$, that is, a flag Bier sphere over a singleton with two ghost vertices added. Let $c=(2,2)$ and $M = \{1, x, x^2, y\}$. Then $M^\vee = \{1, x, y, xy, x^2\}$
and $\Bier_c(M) = \Bier(\ast \sqcup \ast \sqcup \varnothing \sqcup \varnothing)$, that is, a non-flag Bier sphere over two disjoint vertices with two ghost vertices added.
\end{example}

Until the end of this section, we shall discuss the Murai spheres $\Bier_c(M)$, where $M$ is a proper $c$-multicomplex and $c\in\N^m$ with $m=1$. Our first goal is to show that the class of Bier spheres is a proper subclass of the class of Murai spheres. Then we shall explore the relation between Bier spheres and Murai spheres with $m=1$.

\begin{example}\label{MuraiWithMeq1Example}
Let $m=1$ and $c\in\N$. Every proper $c$-multicomplex is of the form $M_a=\{1,\ldots,x^a\}=\langle x^a\rangle$, where $0\le a<c$. Consider the following three cases.

\begin{itemize}
\item If $a=0$, then $\Facet_c(M)=\{\{\gen x{}1,\ldots,\gen x{}{i-1},\gen x{}{i+1},\ldots,\gen x{}c\}\mid 1\le i\le c\}$, so $\Bier_c(M_0)$ is a boundary of a simplex $\Delta_{\{1,\ldots,c\}}$.

\item If $a=c-1$, then $\Facet_c(M)=\{\{\gen x{}0,\ldots,\gen x{}{i-1},\gen x{}{i+1},\ldots,\gen x{}{c-1}\}\mid 0\leq i\leq c-1\}$, so $\Bier_c(M_{c-1})$ is a boundary of a simplex $\Delta_{\{0,\ldots,c-1\}}$.

\item If $1\le a\le c-2$, then $\Facet_c(M)=\{\{\gen x{}0,\ldots,\gen x{}{i-1},\gen x{}{i+1},\ldots,\gen x{}{j-1},\gen x{}{j+1}\ldots,\gen x{}{c}\}\mid 0\le i\le a< j\le c\}$, hence $\Bier_c(M_a)$ is a join of two spheres, one is a boundary of a simplex $\Delta_{\{0,\ldots,a\}}$ and the other is a boundary of a simplex $\Delta_{\{a+1,\ldots,c\}}$.
\end{itemize}
\end{example}

Observe that if $c\in\N, c\geq 3, 1\leq a\leq c-2$, the Murai sphere $\Bier_c(M)$ in the above example is a join of the boundary of a simplex with $a+1$ vertices and the boundary of a simplex with $c-a$ vertices; both numbers could be made arbitrarily large when $c$ is large enough. Hence, any join of two boundaries of simplices is a Murai sphere. On the other hand, one can completely describe the case where a classical Bier sphere is isomorphic to a join of two boundaries of simplices. We obtain the following two results.

\begin{lemma}\label{BierWithCompleteGraphLemma}
Let $K\neq\Delta_{[m]}$ be a simplicial complex with $m\geq 2$. Then the 1-skeleton of $\Bier(K)$ is a complete graph if and only if either $K$ or $K^\vee$ has no vertices.    
\end{lemma}
\begin{proof}
The implication $\Leftarrow$ is clear, so let us prove the implication $\Rightarrow$. If the 1-skeleton of $\Bier(K)$ is a complete graph, then either $i\in\min(K)$, or $i'\in\min(K^\vee)$ for each $i, 1\leq i\leq m$; otherwise $\{i,i'\}\in\min(\Bier(K))$. 

We can assume that $m\in\min(K)$. By Alexander duality, $\{m'\}^c=\{1',2',\ldots,(m-1)'\}\in K^\vee$. Then $i\in\min(K)$ for each $i, 1\leq i\leq m-1$. Thus, $K=\varnothing_{[m]}$, which finishes the proof.
\end{proof}

\begin{proposition}
Let $K\neq\Delta_{[m]}$ be a simplicial complex with $m\geq 2$. Then $\Bier(K)=\partial\Delta^{a}\ast\partial\Delta^{b}$ if and only if one of the following conditions holds:
\begin{itemize}
\item either $a$, or $b$ equals $0$ and either $K=\varnothing_{[m]}$ or $K=\partial\Delta_{[m]}$;
\item either $a\geq 2, b=1$, or $a=1, b\geq 2$ and either $K$ has a unique real vertex or $K=\cone(\partial\Delta_{[m-1]})$;
\item $a=b=1$ and either $K=\langle\{1\}\rangle$, or $K=\langle\{1,2\}\rangle$, or $K=\langle\{1,2\},\{2,3\}\rangle$ with $m=3$.
\end{itemize}    
\end{proposition}
\begin{proof}
Recall that for any simplicial complex $L$, one has:
$$
L\ast\partial\Delta^{0}=L\text{ and }L\ast\partial\Delta^{1}=\Sigma L.
$$
Then \cite[Lemma 3.3]{LZ} shows when a Bier sphere can be isomorphic to a join of a simplicial complex and $\partial\Delta^{d}$ with $d=0$, or $1$. Now suppose $a, b\geq 2$. Then the 1-skeleton of $\Bier(K)$ is a complete graph, and hence, by Lemma~\ref{BierWithCompleteGraphLemma}, $\Bier(K)$ is isomorphic to the boundary of a simplex with $m$ vertices, a contradiction. 
\end{proof}

Note that Lemma~\ref{BierWithCompleteGraphLemma} shows $\Bier(K)$ is \emph{2-neighborly}, that is, any two vertices are linked by an edge, if and only if either $K$ or $K^\vee$ has no vertices; in this case, the 1-skeleton of $\Bier(K)$ is isomorphic to the complete graph $K_m$. In particular, any complete graph can be realized as a 1-skeleton of a Bier and hence a Murai sphere. However, the dimension of such a Bier sphere always equals $m-2$.

Moreover, Example~\ref{MuraiWithMeq1Example} suggests that we could also get a complete graph
$K_{c+1}$ as the 1-skeleton of a Murai sphere if we take $m=1$ and $1\leq a\leq c-2$ with $a+1\geq
3$ and $c-a\geq 3$. It follows that here the dimension of the Murai sphere equals $c-2\geq 3$; the lowest possible case is clearly $a=2, c=5$. 

A question arises naturally: given an integer $d\geq 2$, find all $r$ such that the complete graph $K_r$ arises as the 1-skeleton of a $d$-dimensional Murai sphere. This problem will be addressed in our subsequent publications.

\section{Classification in low dimensions}

In this section, we are going to classify all generalized Bier spheres in dimensions 1 and 2. In what follows, we use the notation of the form $I\sqcup\ast\sqcup\varnothing$ for a simplicial complex on $[4]$ having one ghost vertex, one maximal edge and one maximal vertex, disjoint from each other.

It is easy to classify all 1-dimensional Murai spheres $\Bier_c(M)$; indeed, this is the case if and only if either $c=(3)$, or $c=(2,1)$, or $c=(1,1,1)$. In what follows, we use the notation $Z_p$ for the $p$-cycle graph.

\begin{figure}
\includegraphics[scale=0.65]{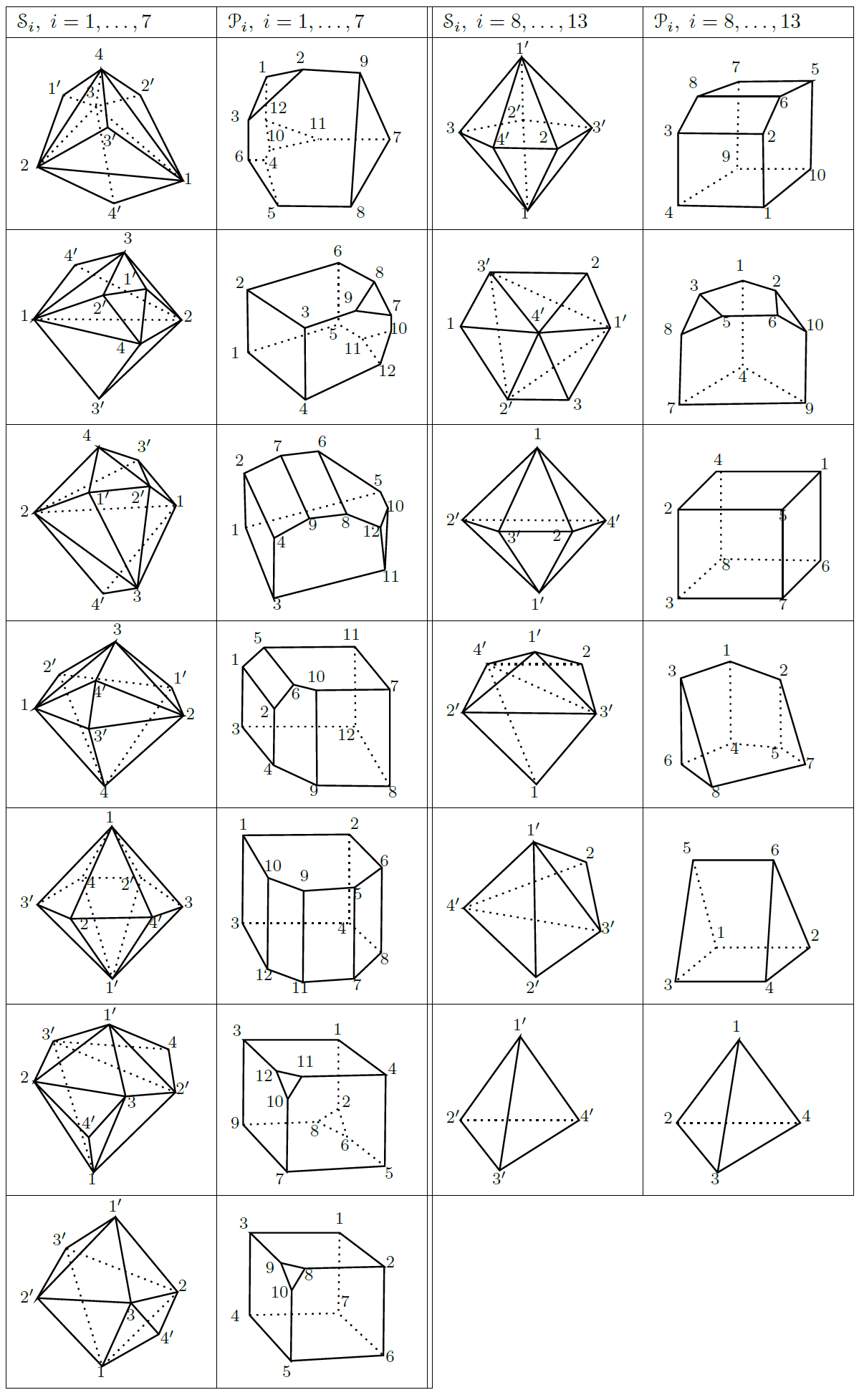}
\caption{The two-dimensional Murai spheres $\mathcal S_i$ with $c=(1,1,1,1)$ and the corresponding simple polytopes $\mathcal P_i$ such that $\mathcal S_i=\partial\mathcal P_i^{*}$ for $1\leq i\leq 13$.}
\label{TwoDimBierFig}
\end{figure}

\begin{example}
Let $c=(3)$. Then one has the following spheres:
\begin{itemize}
\item[(1)] $M=\langle 1\rangle$ and $M^\vee=\langle x^2\rangle$. Then
\(
\Bier_c(M)=\Bier_c(M^\vee)\cong\Bier(\varnothing_{[3]})=Z_3.
\)
\item[(2)] $M=\langle x\rangle=M^\vee$. Then
\(
\Bier_c(M)=\Bier_c(M^\vee)\cong\Bier(\ast\sqcup\varnothing\sqcup\varnothing)=Z_4.
\)
\end{itemize}
Let $c=(2,1)$. Then one has the following spheres:
\begin{itemize}
\item[(1)] $M=\langle 1\rangle$ and $M^\vee=\langle x^2, xy\rangle$. Then 
$\Bier_c(M)=\Bier_c(M^\vee)\cong\Bier(\varnothing_{[3]})=Z_3$. 
\item[(2)] $M=\langle x\rangle$ and $M^\vee=\langle x^2, y\rangle$. Then
$\Bier_c(M)=\Bier_c(M^\vee)\cong\Bier(\ast\sqcup \varnothing\sqcup \varnothing)=Z_4$.
\item[(3)] $M=\langle y\rangle$ and $M^\vee=\langle xy\rangle$. Then 
$\Bier_c(M)=\Bier_c(M^\vee)\cong\Bier(\ast\sqcup\varnothing\sqcup\varnothing)=Z_4$. 
\item[(4)] $M=\langle x, y\rangle=M^\vee$. Then 
$\Bier_c(M)=\Bier_c(M^\vee)\cong\Bier(\ast\sqcup\ast\sqcup\varnothing)=Z_5$.
\item[(5)] $M=\langle x^2\rangle=M^\vee$. Then 
$\Bier_c(M)=\Bier_c(M^\vee)\cong\Bier(\varnothing_{[3]})=Z_3$.
\end{itemize}
Let $c=(1,1,1)$. Then one has the following spheres:
\begin{itemize}
\item[(1)] $M=\langle 1\rangle$ and $M^\vee=\langle xy, xz, yz\rangle$. Then 
$\Bier_c(M)=\Bier_c(M^\vee)\cong\Bier(\varnothing_{[3]})=Z_3$.
\item[(2)] $M=\langle x\rangle$ and $M^\vee=\langle xy, xz\rangle$. Then 
$\Bier_c(M)=\Bier_c(M^\vee)\cong\Bier(\ast\sqcup\varnothing\sqcup\varnothing)=Z_4$.
\item[(3)] $M=\langle xy\rangle=M^\vee$. Then 
$\Bier_c(M)=\Bier_c(M^\vee)\cong\Bier(\ast\sqcup\varnothing\sqcup\varnothing)=Z_4$.
\item[(4)] $M=\langle x, y\rangle$ and $M^\vee=\langle z, xy\rangle$. Then 
$\Bier_c(M)=\Bier_c(M^\vee)\cong\Bier(\ast\sqcup\ast\sqcup\varnothing)=Z_5$.
\item[(5)] $M=\langle x, y, z\rangle=M^\vee$. Then
$\Bier_c(M)=\Bier_c(M^\vee)\cong\Bier(\ast\sqcup\ast\sqcup\ast)=Z_6$.
\end{itemize}
\end{example}

Thus, all possible one-dimensional Murai spheres are listed above. 

\begin{proposition}
The only 1-dimensional Murai spheres are the boundaries of a triangle, a square, a pentagon, or a hexagon.
\end{proposition}

Now we turn to the class of Murai spheres $\Bier_c(M)$ in dimension two; this is the case if and only if either $c=(4)$, or $c=(3,1)$, or $c=(2,2)$, or $c=(2,1,1)$, or $c=(1,1,1,1)$. In what follows, for a convex simple polytope $P$ we consider its \emph{nerve complex} $K_P=\partial P^*$, the boundary of the dual simplicial polytope. Furthermore, if a simple polytope $P$ is combinatorially equivalent to a polytope obtained from the $n$-simplex $\Delta^n$ by consecutively cutting off $k$ vertices for $k\geq 0$ using hyperplanes in general position, we call $P$ a \emph{truncation polytope} of a type in $\vc^k(\Delta^n)$ and write $P\in\vc^k(\Delta^n)$. An $m$-gon $P_m$ is a $2$-dimensional truncation polytope for each $m\geq 3$. Nerve complex of a truncation polytope is called a \emph{stacked sphere}. It bounds a \emph{stacked polytope}, a simplicial $n$-polytope isomorphic to the one obtained by consecutively performing stellar subdivisions in facets (i.e. gluing pyramids over facets), starting from $\Delta^n$. Stacked Bier spheres were classified in~\cite{LV}. We shall consider stacked Murai spheres in the next section.

The case $c=(4)$ is covered by Example~\ref{MuraiWithMeq1Example}. This yields $\partial\Delta^3$ and $\partial\Delta^2\ast\partial\Delta^1$ as Murai spheres.

The case $c=(1,1,1,1)$ of classical Bier spheres was covered in~\cite[Theorem 2.16]{LS}, see Figure~\ref{TwoDimBierFig}.

\begin{example}
Let $c=(3,1)$. Then one has the following spheres:
\begin{itemize}
\item[(1)] $M=\langle 1\rangle$, $M^\vee=\langle x^2y,x^3\rangle$. Then
$\Bier_{c}(M)=\Bier_{c}(M^\vee)\cong\Bier(\varnothing_{[4]})=K_P$, $P=\Delta^3$;
\item[(2)] $M=\langle y\rangle$, $M^\vee=\langle x^2y\rangle$. Then
$\Bier_{c}(M)=\Bier_{c}(M^\vee)\cong\Bier(\cone(Z_3))=K_P$, $P=P_3\times I$;
\item[(3)] $M=\langle x\rangle$, $M^\vee=\langle xy,x^3\rangle$. Then
$\Bier_{c}(M)=\Bier_{c}(M^\vee)\cong\Bier(\cone(Z_3))=K_P$, $P=P_3\times I$;
\item[(4)] $M=\langle xy\rangle=M^\vee$. Then
$\Bier_{c}(M)=\Bier_{c}(M^\vee)\cong\Bier(\cone(Z_4))=K_P$, $P=I^3$;
\item[(5)] $M=\langle x^2\rangle$, $M^\vee=\langle x^3,y\rangle$. Then
$\Bier_{c}(M)=\Bier_{c}(M^\vee)\cong\Bier(\cone(Z_3))=K_P$, $P=P_3\times I$;
\item[(6)] $M=\langle x^3\rangle=M^\vee$. Then
$\Bier_{c}(M)=\Bier_{c}(M^\vee)\cong\Bier(\varnothing_{[4]})=K_P$, $P=\Delta^3$;
\item[(7)] $M=\langle x,y\rangle$, $M^\vee=\langle xy,x^2\rangle$. Then
$\Bier_c(M)=\Bier_c(M^\vee)\cong\Bier(\ast\sqcup\ast\sqcup\varnothing\sqcup\varnothing)=K_P$, $P\in\vc^2(\Delta^3)$;
\item[(8)] $M=\langle x^2,y\rangle=M^\vee$. Then
$\Bier_c(M)=\Bier_c(M^\vee)\cong\Bier(\ast\sqcup\ast\sqcup\varnothing\sqcup\varnothing)=K_P$, $P\in\vc^2(\Delta^3)$;
\end{itemize}
\end{example}

\begin{example}
Let $c=(2,2)$. Then one has the following spheres:
\begin{itemize}
\item[(1)] $M=\langle 1\rangle$, $M^\vee=\langle x^2y, xy^2\rangle$. Then
$\Bier_c(M)=\Bier_c(M^\vee)\cong\Bier(\varnothing_{[4]})=K_P$, $P=\Delta^3$;
\item[(2)] $M=\langle x\rangle$, $M^\vee=\langle x^2y,y^2\rangle$. Then
$\Bier_c(M)=\Bier_c(M^\vee)\cong\Bier(\cone(Z_3))=K_P$, $P=P_3\times I$;
\item[(3)] $M=\langle y\rangle$, $M^\vee=\langle xy^2,x^2\rangle$. Then
$\Bier_c(M)=\Bier_c(M^\vee)\cong\Bier(\cone(Z_3))=K_P$, $P=P_3\times I$;
\item[(4)] $M=\langle x,y\rangle$, $M^\vee=\langle x^2,xy,y^2\rangle$. Then
$\Bier_c(M)=\Bier_c(M^\vee)\cong\Bier(\ast\sqcup\ast\sqcup\varnothing\sqcup\varnothing)=K_P$, $P\in\vc^2(\Delta^3)$;
\item[(5)] $M=\langle xy\rangle$, $M^\vee=\langle x^2,y^2\rangle$. Then
$\Bier_c(M)=\Bier_c(M^\vee)\cong\Bier(\cone(Z_4))=K_P$, $P=I^3$;
\item[(6)] $M=\langle x^2\rangle$, $M^\vee=\langle x^2y\rangle$. Then
$\Bier_c(M)=\Bier_c(M^\vee)\cong\Bier(\cone(Z_3))=K_P$, $P=P_3\times I$;
\item[(7)] $M=\langle y^2\rangle$, $M^\vee=\langle xy^2\rangle$. Then
$\Bier_c(M)=\Bier_c(M^\vee)\cong\Bier(\cone(Z_3))=K_P$, $P=P_3\times I$;
\item[(8)] $M=\langle x,y^2\rangle$, $M^\vee=\langle xy,y^2\rangle$. Then
$\Bier_c(M)=\Bier_c(M^\vee)\cong\Bier(\ast\sqcup\ast\sqcup\varnothing\sqcup\varnothing)=K_P$, $P\in\vc^2(\Delta^3)$;
\item[(9)] $M=\langle x^2,y\rangle$, $M^\vee=\langle xy,x^2\rangle$. Then
$\Bier_c(M)=\Bier_c(M^\vee)\cong\Bier(\ast\sqcup\ast\sqcup\varnothing\sqcup\varnothing)=K_P$, $P\in\vc^2(\Delta^3)$.
\end{itemize}
\end{example}

\begin{example}
Let $c=(2,1,1)$. Then one has the following spheres:
\begin{itemize}
\item[(1)] $M=\langle 1\rangle$, $M^\vee=\langle x^2y,x^2z,xyz\rangle$. Then
$\Bier_c(M)=\Bier_c(M^\vee)\cong\Bier(\varnothing_{[4]})=K_P$, $P=\Delta^{3}$;
\item[(2)] $M=\langle x\rangle$, $M^\vee=\langle x^2y,x^2z,yz\rangle$. Then
$\Bier_c(M)=\Bier_c(M^\vee)\cong\Bier(\cone(Z_3))=K_P$, $P=P_3\times I$;
\item[(3)] $M=\langle y\rangle$, $M^\vee=\langle xyz,x^2y\rangle$. Then
$\Bier_c(M)=\Bier_c(M^\vee)\cong\Bier(\cone(Z_3))=K_P$, $P=P_3\times I$;
\item[(4)] $M=\langle z\rangle$, $M^\vee=\langle xyz,x^2z\rangle$. Then
$\Bier_c(M)=\Bier_c(M^\vee)\cong\Bier(\cone(Z_3))=K_P$, $P=P_3\times I$;
\item[(5)] $M=\langle xy\rangle$, $M^\vee=\langle x^2y,yz\rangle$. Then
$\Bier_c(M)=\Bier_c(M^\vee)\cong\Bier(\cone(Z_4))=K_P$, $P=I^3$;
\item[(6)] $M=\langle xz\rangle$, $M^\vee=\langle x^2z,yz\rangle$. Then
$\Bier_c(M)=\Bier_c(M^\vee)\cong\Bier(\cone(Z_4))=K_P$, $P=I^3$;
\item[(7)] $M=\langle yz\rangle$, $M^\vee=\langle xyz\rangle$. Then
$\Bier_c(M)=\Bier_c(M^\vee)\cong\Bier(\cone(Z_4))=K_P$, $P=I^3$;
\item[(8)] $M=\langle x^2\rangle$,  $M^\vee=\langle x^2y,x^2z\rangle$. Then
$\Bier_c(M)=\Bier_c(M^\vee)\cong\Bier(\cone(Z_3))=K_P$, $P=P_3\times I$;
\item[(9)] $M=\langle x^2y\rangle=M^\vee$. Then
$\Bier_c(M)=\Bier_c(M^\vee)=\Bier(\cone(Z_3))=K_P$, $P=P_3\times I$;
\item[(10)] $M=\langle x^2z\rangle=M^\vee$. Then
$\Bier_c(M)=\Bier_c(M^\vee)\cong\Bier(\cone(Z_3))=K_P$, $P=P_3\times I$;
\item[(11)] $M=\langle x, y\rangle$, $M^\vee=\langle yz,x^2y,xz\rangle$. Then
$\Bier_c(M)=\Bier_c(M^\vee)\cong\Bier(\ast\sqcup\ast\sqcup\varnothing\sqcup\varnothing)=K_P$, $P\in\vc^2(\Delta^3)$;
\item[(12)] $M=\langle x, z\rangle$, $M^\vee=\langle xy,yz,x^2z\rangle$. Then
$\Bier_c(M)=\Bier_c(M^\vee)\cong\Bier(\ast\sqcup\ast\sqcup\varnothing\sqcup\varnothing)=K_P$, $P\in\vc^2(\Delta^3)$;
\item[(13)] $M=\langle y, z\rangle$, $M^\vee=\langle xyz,x^2\rangle$. Then
$\Bier_c(M)=\Bier_c(M^\vee)\cong\Bier(\ast\sqcup\ast\sqcup\varnothing\sqcup\varnothing)=K_P$, $P\in\vc^2(\Delta^3)$;
\item[(14)] $M=\langle xy, z\rangle$, $M^\vee=\langle x^2,yz,xy\rangle$. Then
$\Bier_c(M)=\Bier_c(M^\vee)\cong\Bier(I\sqcup\ast\sqcup\varnothing)=K_P$, $P\in\vc^1(I^3)$;
\item[(15)] $M=\langle xz, y\rangle$, $M^\vee=\langle x^2,yz,xz\rangle$. Then
$\Bier_c(M)=\Bier_c(M^\vee)\cong\Bier(I\sqcup\ast\sqcup\varnothing)=K_P$, $P\in\vc^1(I^3)$;
\item[(16)] $M=\langle x^2,y\rangle$, $M^\vee=\langle x^2y,xz\rangle$. Then
$\Bier_c(M)=\Bier_c(M^\vee)\cong\Bier(\ast\sqcup\ast\sqcup\varnothing\sqcup\varnothing)=K_P$, $P\in\vc^2(\Delta^3)$;
\item[(17)] $M=\langle x^2,z\rangle$, $M^\vee=\langle x^2z,xy\rangle$. Then
$\Bier_c(M)=\Bier_c(M^\vee)\cong\Bier(\ast\sqcup\ast\sqcup\varnothing\sqcup\varnothing)=K_P$, $P\in\vc^2(\Delta^3)$;
\item[(18)] $M=\langle x^2y,z\rangle$, $M^\vee=\langle x^2,xy\rangle$. Then
$\Bier_c(M)=\Bier_c(M^\vee)\cong\Bier(\ast\sqcup\ast\sqcup\varnothing\sqcup\varnothing)=K_P$, $P\in\vc^2(\Delta^3)$;
\item[(19)] $M=\langle x^2z,y\rangle$, $M^\vee=\langle x^2,xz\rangle$. Then
$\Bier_c(M)=\Bier_c(M^\vee)\cong\Bier(\ast\sqcup\ast\sqcup\varnothing\sqcup\varnothing)=K_P$, $P\in\vc^2(\Delta^3)$;
\item[(20)] $M=\langle x,yz\rangle$, $M^\vee=\langle xz,xy,yz\rangle$. Then
$\Bier_c(M)=\Bier_c(M^\vee)\cong\Bier(I\sqcup\ast\sqcup\varnothing)=K_P$, $P\in\vc^1(I^3)$;
\item[(21)] $M=\langle xy,yz\rangle=M^\vee$. Then
$\Bier_c(M)=\Bier_c(M^\vee)\cong\Bier(\cone(Z_5))=K_P$, $P=P_5\times I$;
\item[(22)] $M=\langle xz,yz\rangle=M^\vee$. Then
$\Bier_c(M)=\Bier_c(M^\vee)\cong\Bier(\cone(Z_5))=K_P$, $P=P_5\times I$;
\item[(23)] $M=\langle xy, xz\rangle$, $M^\vee=\langle x^2, yz\rangle$. Then
$\Bier_c(M)=\Bier_c(M^\vee)\cong\Bier(\cone(Z_5))=K_P$, $P=P_5\times I$;
\end{itemize}
Finally, unlike what took place before, there are 3 different combinatorial types of Murai spheres corresponding to the set of combinatorial truncation polytopes $\vc^3(\Delta^3)$:
\begin{itemize}
\item[(24)] $M=\langle x, y, z\rangle$, $M^\vee=\langle x^2, xy, xz, yz\rangle$. Then
$\Bier_c(M)=\Bier_c(M^\vee)\cong\Bier(\ast\sqcup\ast\sqcup\ast\sqcup\varnothing)=K_P$, where $P\in\vc^3(\Delta^3)$. In the corresponding stacked polytope $P^*$, the three pyramids are glued along the three facets of the simplex $\Delta^3$.
\item[(25)] $M=\langle x^2, y, z\rangle$, $M^\vee=\langle x^2, xy, xz\rangle$. Then
$\Bier_c(M)=\Bier_c(M^\vee)=K_P$, where $P\in\vc^3(\Delta^3)$.
In the corresponding stacked polytope $P^*$, the three pyramids have a common edge. Therefore, $P$ is \emph{not} isomorphic to any Bier polytope.
\item[(26)] $M=\langle x^2, xy, z\rangle=M^\vee$. Then
$\Bier_c(M)=\Bier_c(M^\vee)=K_P$, where $P\in\vc^3(\Delta^3)$.
In the corresponding stacked polytope $P^*$, the three pyramids have a common vertex, two of the three pairs of these pyramids have a common edge, while the last pair has merely a common vertex. Therefore, $P$ is {\emph{not}} isomorphic to any Bier polytope.
\end{itemize}
\end{example}

Summarizing the observations of this section, we obtain the following classification result.

\begin{theorem}\label{MuraiClassificationTheorem}
Each 1-dimensional Murai sphere is isomorphic to a Bier sphere. Each 2-dimensional Murai sphere is isomorphic to a Bier sphere, see Figure~\ref{TwoDimBierFig}, or to $K_P:=\Bier_{(2,1,1)}(\langle x^2,y,z\rangle)$, or to $K_Q:=\Bier_{(2,1,1)}(\langle x^2,xy,z\rangle)$, where the 3-dimensional truncation polytopes $P$ and $Q$ are shown in Figure~\ref{fig:types}. In particular, all two- and three-dimensional simple polytopes $P$ such that $K_P$ is isomorphic to a Murai sphere have Delzant realizations.
\end{theorem}

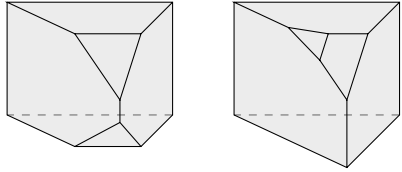
\begin{figure}[ht]
\begin{tikzpicture}
\pgfmathsetmacro\r{0.4}
\path 
    coordinate(A) at (0,0)
    coordinate(B) at (0.7,0.7)
    coordinate(C) at (-1.5,0.7)
    coordinate(A1) at ($(A)+(0,1.5)$)
    coordinate(B1) at ($(B)+(0,1.5)$)
    coordinate(C1) at ($(C)+(0,1.5)$)
    coordinate(AA) at ($(A1)!\r!(A)$)
    coordinate(AB) at ($(A1)!\r!(B1)$)
    coordinate(AC) at ($(A1)!\r!(C1)$)
    coordinate(CA) at ($(C1)!\r!(A1)$)
    coordinate(CB) at ($(C1)!\r!(B1)$)
    coordinate(CC) at ($(C1)!\r!(C)$)
    coordinate(XA) at ($(A)!\r!(A1)$)
    coordinate(XB) at ($(A)!\r!(B)$)
    coordinate(XC) at ($(A)!\r!(C)$)
    coordinate(YA) at ($(XA)!\r!(A)$)
    coordinate(YB) at ($(AB)!\r!(XA)$)
    coordinate(YC) at ($(AC)!\r!(XA)$)
    coordinate(ZA) at ($(C)!\r!(A)$)
    coordinate(ZB) at ($(C)!\r!(B)$)
    coordinate(ZC) at ($(C)!\r!(C1)$)
    coordinate(WA) at ($(AC)!\r!(AA)$)
    coordinate(WB) at ($(AC)!\r!(AB)$)
    coordinate(WC) at ($(AC)!0.2!(C1)$);
\clip (-1.6,-0.1) rectangle (0.8,2.3);    
\draw[dashed] (B)--(C);
\fill[gray!30,opacity=0.5] (XC)--(XB)--(B)--(B1)--(C1)--(C)--cycle;
\draw (C1)--(C)--(XC)--(XB)--(B)--(B1)--(C1)--(AC)--(AB)--(AA)--(AC) (AB)--(B1) (AA)--(XA)--(XB) (XC)--(XA);
\end{tikzpicture}
\hskip5mm
\begin{tikzpicture}
\clip (-1.6,-0.1) rectangle (0.8,2.3);    
\draw[dashed] (B)--(C);
\fill[gray!30,opacity=0.5] (A)--(B)--(B1)--(C1)--(C)--cycle;
\draw (C1)--(C)--(A)--(B)--(B1)--(C1)--(WC)--(WA)--(WB)--(AB)--(AA)--(WA) (AB)--(B1) (A)--(AA) (WC)--(WB);
\end{tikzpicture}
\caption{Truncation polytopes $P$ and $Q$ corresponding to the two exceptional Murai spheres in dimension two.}\label{fig:types}
\end{figure}

\section{Chordality}

In this section, we classify all chordal Murai spheres $\Bier_c(M)$ with $c\in\N^m$ and $m\leq 2$. Due to~\cite[Theorem 8.5]{Kalai}, stackedness implies chordality, see also~\cite[Proposition 4.5]{LV}. Finally, we shall find all possible pairs $(k,n)$ of positive integers such that a Murai sphere is a nerve complex of a truncation polytope from the set $\vc^k(\Delta^n)$. We start with the definition of chordality.

\begin{definition}
A simple graph $\Gamma=(V,E)$ is called \emph{chordal} if it has no induced subgraphs isomorphic to cycles of length greater than three. A simplicial complex is \emph{chordal} if its 1-skeleton is a chordal graph.
\end{definition}

\begin{example}
Let $m=1$ and $c\in\N$. By \Cref{MuraiWithMeq1Example}, every Murai sphere $\Bier_c(M)$ is a boundary of a simplex or a join of two boundaries of simplices. The boundary of every simplex is obviously chordal. A join of boundaries of simplices is chordal if and only if at least one of the simplices is of dimension at least 2.
Hence $\Bier_c(M)$ is not chordal if and only if $c=3$ and $M=\langle x\rangle$.
\end{example}

\begin{example}
Let $c=(c_1,\ldots,c_m)$, $c_i\ge 2$ and $i\ne j$. For $M=\langle x_ix_j\rangle$ the Murai sphere $\Bier_c(M)$ is not chordal. Indeed, since $M=\{x_i^0x_j^0,x_i^0x_j^1,x_i^1x_j^0,x_i^1x_j^1\}$, we have that 1-simplex $\{\gen{x}{k}{0},\gen{x}{k}{1}\}$, $k=i,j$, is not a subset of any maximal simplex $G(x_i^ax_j^b;x_k^d)$ of $\Bier_c(M)$. 
As $\{\gen{x}{i}{0},\gen{x}{j}{0}\}\subset G(x_i^1x_j^1;x_i^2)$, $\{\gen{x}{i}{0},\gen{x}{j}{1}\}\subset G(x_i^1x_j^0;x_i^2)$, $\{\gen{x}{i}{1},\gen{x}{j}{0}\}\subset G(x_i^0x_j^1;x_i^2)$, and $\{\gen{x}{i}{1},\gen{x}{j}{1}\}\subset G(x_i^0x_j^0;x_i^2)$, the vertices $\gen{x}{i}{0},\gen{x}{j}{0}, \gen{x}{i}{1},\gen{x}{j}{1}$ form a 4-cycle in $\Bier_c(M)$ without a chord, hence $\Bier_c(M)$ is not chordal.
\end{example}

\begin{proposition}\label{ChordalSpheresWithSecondCorrdinateOneProp}
Suppose $c=(c_1,c_2)\in\N^2$, where $c_1\geq 2$ and $c_2=1$.
Then a Murai sphere $\Bier_c(M)$ is not chordal if and only if one of the following conditions holds:
\begin{itemize}
\item $c=(2,1)$ and either $M=\langle x_1\rangle$, or $M^\vee=\langle x_1\rangle$;
\item $c=(2,1)$ and either $M=\langle x_1,x_2\rangle$, or $M^\vee=\langle x_1,x_2\rangle$;
\item $c_1\geq 2$ is arbitrary and either $M=\langle x_1x_2\rangle$, or $M^\vee=\langle x_1x_2\rangle$.
\end{itemize}
\end{proposition}
\begin{proof}
In what follows an empty product is set to be equal to $1$. We consider all possible cases below.

If $M=\langle x_1^k\rangle$, where $0\leq k\leq c_1$, then $M^\vee=\langle (x_1^{k+1})^c,(x_2)^c\rangle=\langle x_1^{c_1-k-1}x_2,x_1^{c_1}\rangle$ and we obtain $\pol(I_c(M))=(x_{1,0}\ldots x_{1,k},x_{2,0})$ and $\pol^*(I_c(M^\vee))=(x_{1,c_1}\ldots x_{1,k+1}x_{2,1})$. Therefore, $\Bier_c(M)=\partial\Delta^{k}\ast\partial\Delta^{c_1-k}$. It is chordal, unless $k=1=c_1-k$; that is, $c=(2,1), M=\langle x_1\rangle$ and $\Bier_c(M)=Z_4$.

If $M=\langle x_1^{a}, x_1^{b}x_{2}\rangle$ and $c_1\geq a>b\geq 0$, then $M^\vee=\langle (x_1^{a+1})^c,(x_1^{b+1}x_2)^c\rangle=\langle x_1^{c_1-a-1}x_2, x_1^{c_1-b-1}\rangle$ and one gets $\pol(I_c(M))=(x_{1,0}\ldots x_{1,a},x_{1,0}\ldots x_{1,b}x_{2,0})$ and $\pol^*(I_c(M^\vee))=(x_{1,c_1}\ldots x_{1,a+1}x_{2,1},x_{1,c_1}\ldots x_{1,b+1})$. Therefore, the Stanley--Reisner ideal of $\Bier_c(M)$ equals:
$$
I=(x_{1,0}\ldots x_{1,a},x_{1,0}\ldots x_{1,b}x_{2,0},x_{1,c_1}\ldots x_{1,a+1}x_{2,1},x_{1,c_1}\ldots x_{1,b+1},x_{2,0}x_{2,1}).
$$

If $b\geq 1$ and $c_1\geq b+3$, then there is only one non-edge in the Murai sphere, hence the sphere is chordal.

If $b=0$, then $\{\gen{x}{1}{0},\gen{x}{2}{0}\}$ and $\{\gen{x}{2}{0},\gen{x}{2}{1}\}$ are non-edges and $\{\gen{x}{1}{0},\gen{x}{2}{1}\}$ is an edge. To have a chordless cycle we need at least two more vertices, $\gen{x}{1}{i}$ and $\gen{x}{1}{j}$, not connected to each other. This implies $c_1=2$, and hence $i=1$ and $j=2$. If $a=2$, then $\Bier_c(M)$ would equal the 4-cycle with one chord, and if $a=1$ we obtain a chordless 5-cycle as 
$c=(2,1), M=\langle x_1, x_2\rangle$ and $\Bier_c(M)=Z_5$ on the vertex set $\{\gen{x}{1}{0},\gen{x}{1}{2},\gen{x}{2}{0},\gen{x}{1}{1},\gen{x}{2}{1}\}$.

If $c_1=b+1$, then $a=c_1$ and we get 
$I=(x_{1,0}\ldots x_{1,b}x_{2,0},x_{2,1},x_{1,b+1})$. Therefore, $\Bier_c(M)=\partial\Delta^{b+1}$, is chordal.

If $c_1=b+2$ and $a=c_1$, then $I=(x_{1,0}\ldots x_{1,b}x_{2,0},x_{2,1},x_{1,b+2}x_{1,b+1})$. We have at least two non-edges if and only if $b=0$; the non-edges are $\{x_{1}^{0},x_{2}^{0}\}$ and $\{x_{1}^{1},x_{1}^{2}\}$. We obtain a chordless 4-cycle as 
$c=(2,1), M=\langle x_1^2, x_2\rangle$ and $\Bier_c(M)=Z_4$ on the vertex set 
$\{\gen{x}{1}{0},\gen{x}{2}{0},\gen{x}{1}{1},\gen{x}{1}{2}\}$,
which has already been considered above, since $M^\vee=\langle x_1\rangle$.

If $c_1=b+2$ and $a=c_1-1=b+1$, then the Stanley--Reisner ideal of $\Bier_c(M)$ has the form:
$I=(x_{1,0}\ldots x_{1,b+1},x_{1,0}\ldots x_{1,b}x_{2,0},x_{1,b+2}x_{2,1},x_{1,b+2}x_{1,b+1},x_{2,0}x_{2,1})$.

In case $b\geq 1$, we have 3 non-edges: $\{\gen{x}{1}{b+2},\gen{x}{2}{1}\}, \{\gen{x}{1}{b+2},\gen{x}{1}{b+1}\}$, and $\{\gen{x}{2}{0},\gen{x}{2}{1}\}$. No chordless 4-cycle corresponds to any pair of those.

In case $b=0$, we have 
$I=(x_{1,0}x_{1,1},x_{1,0}x_{2,0},x_{1,2}x_{2,1},x_{1,2}x_{1,1},x_{2,0}x_{2,1})$, which yields a chordless 5-cycle as 
$c=(2,1),M=\langle x_1, x_2\rangle$ and $\Bier_c(M)=Z_5$ on the vertex set 
$\{\gen{x}{1}{0},\gen{x}{1}{2},\gen{x}{2}{0},\gen{x}{1}{1},\gen{x}{2}{1}\}$. This has already been considered above. 

If $M=\langle x_1^{k}x_2\rangle$ and $0\leq k\leq c_1-1$, then $M^\vee=\langle (x_1^{k+1})^c\rangle=\langle x_1^{c_1-k-1}x_2\rangle$ and we obtain 
$\pol(I_c(M))=(x_{1,0}\ldots x_{1,k})$ and $\pol^*(I_c(M^\vee))=(x_{1,c_1}\ldots x_{1,k+1})$. Therefore, the Stanley--Reisner ideal of $\Bier_c(M)$ equals $I=(x_{1,0}\ldots x_{1,k},x_{1,c_1}\ldots x_{1,k+1},x_{2,0}x_{2,1})$. Thus, $\Bier_c(M)=\partial\Delta^{k}\ast\partial\Delta^{c_1-k-1}\ast\partial\Delta^1$.
The latter sphere is non-chordal if and only if $k=1$, or $c_1-k-1=1$.
This finishes the proof.
\end{proof}

\begin{lemma}\label{BoundsChordlessCycleLemma}
If $Z$ is a chordless cycle in $\Bier_c(M)$ with $c\in\N^m$, then the length of $Z$ is at most $m+3$.   
\end{lemma}
\begin{proof}
Let $\{\gen{x}{i_1}{\alpha_1},\gen{x}{i_2}{\alpha_2}\}$ be an edge in $Z$. Then there exists a facet $G(x_1^{\beta_1}\cdots x_m^{\beta_m};x_j^\gamma)$ in the Murai sphere $\Bier_c(M)$ containing this edge. Therefore, both vertices $\gen{x}{i_1}{\alpha_1}$ and $\gen{x}{i_2}{\alpha_2}$ are connected with all vertices of the Murai sphere except maybe $\gen{x}{1}{\beta_1}, \cdots, \gen{x}{m}{\beta_m},\gen{x}{j}{\gamma}$. Since $Z$ is chordless, we have that the set of vertices of $Z$ is a subset of $\{\gen{x}{i_1}{\alpha_1},\gen{x}{i_2}{\alpha_2},\gen{x}{1}{\beta_1}, \cdots, \gen{x}{m}{\beta_m},\gen{x}{j}{\gamma}\}$. This finishes the proof.
\end{proof}

As was shown in the proof of Proposition~\ref{ChordalSpheresWithSecondCorrdinateOneProp}, the Murai sphere $\Bier_{(2,1)}(\langle x_1,x_2\rangle)$ has an induced chordless 5-cycle. The next result asserts that this can not be the case when $c_2\geq 2$.

\begin{lemma}\label{Chordless4CycleLemma}
Let $c=(c_1,c_2)\in\N^2$ and $c_1,c_2\ge 2$. If $Z$ is a chordless cycle in a Murai sphere $\Bier_c(M)$, then the length of $Z$ is equal to 4.    
\end{lemma}
\begin{proof}
By \Cref{BoundsChordlessCycleLemma}, the length of a chordless cycle in $\Bier_c(M)$ is at most 5. Let $Z$ be an induced chordless cycle of length 5 in the Murai sphere. Without loss of generality, we may assume that at least three vertices of $Z$ are of the form $\gen{x}{1}{*}$. Since $c_2\ge 2$, there exists $\beta\le c_2$ such that $\gen{x}{2}{\beta}$ is not a vertex of $Z$. 

If $x_1^{c_1} x_2^\beta\in M$, then $\beta<c_2$ as $M$ is a proper $c$-multicomplex, and hence, by definition of a Murai sphere, $G(x_1^{c_1}x_2^\beta;x_2^{c_2})$ is a facet of $\Bier_c(M)$. It contains all vertices of the Murai sphere except $\gen{x}{1}{c_1}$, $\gen{x}{2}{\beta}$, and $\gen{x}{2}{c_2}$.
Therefore, it contains at least three vertices of $Z$, so $Z$ is not chordless.

If $x_1^{0} x_2^\beta\not\in M$, then $G(x_1^{0}x_2^0;x_2^{\beta})$ is a facet of $\Bier_c(M)$. It contains at least three vertices of $Z$, hence the cycle $Z$ is not chordless.

It follows that there exists $\alpha<c_1$ such that $x_1^\alpha x_2^\beta\in M$, but $x_1^{\alpha+1} x_2^\beta\not\in M$. Then $G(x_1^{\alpha}x_2^\beta;x_1^{\alpha+1})$ is a facet containing at least three vertices of $Z$, hence $Z$ is not chordless, a contradiction. Thus, we have shown that there is no chordless cycle of length 5 in the Murai sphere $\Bier_c(M)$.
\end{proof}

\begin{proposition}\label{ChordalSpheresWithCoordinatesGreaterThanTwoProp}
Let $c=(c_1,c_2)\in\N^2$ and $c_1,c_2\ge 2$. Then $\Bier_c(M)$ is not chordal if and only if one of the following conditions takes place:
\begin{itemize}
\item either $M=\langle x_1 x_2\rangle$, or $M^\vee=\langle x_1x_2\rangle$;
\item $c_1=2$ and either $M=\langle x_1 x_2^{c_2-1}\rangle$, or $M^\vee=\langle x_1 x_2^{c_2-1}\rangle$;
\item $c_1=3$ and $M=\langle x_1 x_2^{c_2}\rangle=M^\vee$.
\end{itemize}
\end{proposition}
\begin{proof}
Suppose that $\Bier_c(M)$ is not chordal. By \Cref{Chordless4CycleLemma}, the latter is the case if and only if the Murai sphere has a chordless cycle of length 4. Up to symmetries, only the following cases can occur:

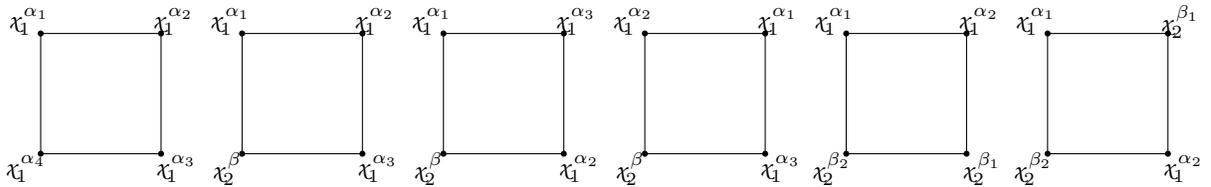
\begin{figure}[h]
\begin{tikzpicture}[scale=0.8]
    \clip (-0.6,-0.6) rectangle (2.6,2.6);
    \draw (0,0) rectangle (2,2);
    \fill (0,0) circle (0.05) (2,0) circle (0.05) (0,2) circle (0.05) (2,2) circle (0.05);
    \draw (0,2)+(-0.2,0.2) node{$\gen{x}{1}{\alpha_1}$}
          (2,2)+(0.2,0.2) node{$\gen{x}{1}{\alpha_2}$}
          (2,0)+(0.25,-0.25) node{$\gen{x}{1}{\alpha_3}$}
          (0,0)+(-0.25,-0.25) node{$\gen{x}{1}{\alpha_4}$};
\end{tikzpicture}
\begin{tikzpicture}[scale=0.8]
    \clip (-0.6,-0.6) rectangle (2.6,2.6);
    \draw (0,0) rectangle (2,2);
    \fill (0,0) circle (0.05) (2,0) circle (0.05) (0,2) circle (0.05) (2,2) circle (0.05);
    \draw (0,2)+(-0.2,0.2) node{$\gen{x}{1}{\alpha_1}$}
          (2,2)+(0.2,0.2) node{$\gen{x}{1}{\alpha_2}$}
          (2,0)+(0.25,-0.25) node{$\gen{x}{1}{\alpha_3}$}
          (0,0)+(-0.25,-0.25) node{$\gen{x}{2}{\beta}$};
\end{tikzpicture}   
\begin{tikzpicture}[scale=0.8]
    \clip (-0.6,-0.6) rectangle (2.6,2.6);
    \draw (0,0) rectangle (2,2);
    \fill (0,0) circle (0.05) (2,0) circle (0.05) (0,2) circle (0.05) (2,2) circle (0.05);
    \draw (0,2)+(-0.2,0.2) node{$\gen{x}{1}{\alpha_1}$}
          (2,2)+(0.2,0.2) node{$\gen{x}{1}{\alpha_3}$}
          (2,0)+(0.25,-0.25) node{$\gen{x}{1}{\alpha_2}$}
          (0,0)+(-0.25,-0.25) node{$\gen{x}{2}{\beta}$};
\end{tikzpicture}   
\begin{tikzpicture}[scale=0.8]
    \clip (-0.6,-0.6) rectangle (2.6,2.6);
    \draw (0,0) rectangle (2,2);
    \fill (0,0) circle (0.05) (2,0) circle (0.05) (0,2) circle (0.05) (2,2) circle (0.05);
    \draw (0,2)+(-0.2,0.2) node{$\gen{x}{1}{\alpha_2}$}
          (2,2)+(0.2,0.2) node{$\gen{x}{1}{\alpha_1}$}
          (2,0)+(0.25,-0.25) node{$\gen{x}{1}{\alpha_3}$}
          (0,0)+(-0.25,-0.25) node{$\gen{x}{2}{\beta}$};
\end{tikzpicture}
\begin{tikzpicture}[scale=0.8]
    \clip (-0.6,-0.6) rectangle (2.6,2.6);
    \draw (0,0) rectangle (2,2);
    \fill (0,0) circle (0.05) (2,0) circle (0.05) (0,2) circle (0.05) (2,2) circle (0.05);
    \draw (0,2)+(-0.2,0.2) node{$\gen{x}{1}{\alpha_1}$}
          (2,2)+(0.2,0.2) node{$\gen{x}{1}{\alpha_2}$}
          (2,0)+(0.25,-0.25) node{$\gen{x}{2}{\beta_1}$}
          (0,0)+(-0.25,-0.25) node{$\gen{x}{2}{\beta_2}$};
\end{tikzpicture}
\begin{tikzpicture}[scale=0.8]
    \clip (-0.6,-0.6) rectangle (2.6,2.6);
    \draw (0,0) rectangle (2,2);
    \fill (0,0) circle (0.05) (2,0) circle (0.05) (0,2) circle (0.05) (2,2) circle (0.05);
    \draw (0,2)+(-0.2,0.2) node{$\gen{x}{1}{\alpha_1}$}
          (2,2)+(0.2,0.2) node{$\gen{x}{2}{\beta_1}$}
          (2,0)+(0.25,-0.25) node{$\gen{x}{1}{\alpha_2}$}
          (0,0)+(-0.25,-0.25) node{$\gen{x}{2}{\beta_2}$};
\end{tikzpicture}
\caption{Those are up to symmetry all possible cases for 4-cycle in a Murai sphere. In cases 2, 3 and 4, we have $\alpha_1<\alpha_2<\alpha_3$.}
\end{figure}

(1) The chordless cycle $Z$ has vertices $\gen{x}{i}{\alpha_1}$, $\gen{x}{i}{\alpha_2}$, $\gen{x}{i}{\alpha_3}$, $\gen{x}{i}{\alpha_4}$. Without loss of generality, we may assume that $i=1$. Moreover, application of Theorem~\ref{FaceIdealMuraiThm} shows that $Z=\{\gen{x}{1}{0}, \gen{x}{1}{c_1}, \gen{x}{1}{1}, \gen{x}{1}{c_1-1}\}$ with $c_1-1\geq 2$, labeled clockwise. This is equivalent to $x_1^2\in I_c(M)\cap I_c(M^\vee)$, which in turn is equivalent to 
$x_1^{c_1-2}x_2^{c_2}\in\max(M)$ and $x_1^2\in I_c(M)$.
It follows that $c_1=3$ and $M=\langle x_1x_2^{c_2}\rangle=M^\vee$.

(2) The chordless cycle $Z$ has vertices $\gen{x}{1}{\alpha_1}$, $\gen{x}{1}{\alpha_2}$, $\gen{x}{1}{\alpha_3}$, $\gen{x}{2}{\beta}$, and $\alpha_1<\alpha_2<\alpha_3$. Since $\{\gen{x}{1}{\alpha_1},\gen{x}{2}{\beta}\}$ is an edge and 
$\{\gen{x}{1}{\alpha_1},\gen{x}{1}{\alpha_3}\}$, $\{\gen{x}{1}{\alpha_2},\gen{x}{2}{\beta}\}$ are not edges, there is a facet containing $\gen{x}{1}{\alpha_1}$, $\gen{x}{2}{\beta}$ and not containing $\gen{x}{1}{\alpha_2}$ and $\gen{x}{1}{\alpha_3}$. The facet is $G(x_1^{\alpha_2}x_2^\delta;x_1^{\alpha_3})$ for some $\delta\ne\beta$. Then $G(x_1^{\alpha_1}x_2^\delta;x_1^{\alpha_3})$ is a facet and contains $\gen{x}{1}{\alpha_2}$ and $\gen{x}{2}{\beta}$, so $Z$ is not chordless.

(3) The chordless cycle $Z$ has vertices $\gen{x}{1}{\alpha_1}$, $\gen{x}{1}{\alpha_3}$, $\gen{x}{1}{\alpha_2}$, $\gen{x}{2}{\beta}$, and $\alpha_1<\alpha_2<\alpha_3$. Using Theorem~\ref{FaceIdealMuraiThm}, it is easy to see that $Z=\{\gen{x}{1}{0}, \gen{x}{1}{c_1}, \gen{x}{1}{1}, \gen{x}{2}{c_2}\}$ with $c_1\geq 2$, labeled clockwise. This is equivalent to $x_1^2\in I_c(M)$ and $x_1x_2\in I_c(M^\vee)$. It follows that this is the case if and only if $c_1=2$ and $x_1^2\in I_c(M), x_1x_2^{c_2-1}\in\max(M)$. If $M=\langle x_1x_2^{c_2-1}, x_2^{c_2}\rangle$, then $M^\vee=\langle x_1, x_2^{c_2}\rangle$, hence $I_c(M^\vee)=(x_1x_2, x_1^2)$, a contradiction with $\{\gen{x}{1}{1}, \gen{x}{1}{2}\}$ being an edge. Therefore, the only possibility is $M=\langle x_1x_2^{c_2-1}\rangle$ and $c_1=2$.

(4) The chordless cycle $Z$ has vertices $\gen{x}{1}{\alpha_2}$, $\gen{x}{1}{\alpha_1}$, $\gen{x}{1}{\alpha_3}$, $\gen{x}{2}{\beta}$, and $\alpha_1<\alpha_2<\alpha_3$. Using Theorem~\ref{FaceIdealMuraiThm}, it is easy to see that $Z=\{\gen{x}{2}{0}, \gen{x}{1}{c_1-1}, \gen{x}{1}{0}, \gen{x}{1}{c_1}\}$ with $c_1\geq 2$, labeled clockwise. Similarly to the previous case, we deduce $c_1=2$ and $M^\vee=\langle x_1x_2^{c_2-1}\rangle$.

(5) The chordless cycle $Z$ has vertices $\gen{x}{1}{\alpha_2}$, $\gen{x}{1}{\alpha_1}$, $\gen{x}{2}{\beta_1}$, $\gen{x}{2}{\beta_2}$, and $\alpha_1<\alpha_2$. Since $\{\gen{x}{2}{\beta_1},\gen{x}{2}{\beta_2}\}$ is an edge in a chordless cycle, there exists $\delta$ such that $\gen{x}{2}{\beta_1},\gen{x}{2}{\beta_2}\in G(x_1^{\alpha_1}x_2^\delta;x_1^{\alpha_2})$. Since $c_1\ge 2$, there is $\gamma\in\{0,\ldots,c_2\}\setminus\{\alpha_1,\alpha_2\}$. If $x_1^\gamma x_2^\delta\in M$ then $\gen{x}{1}{\alpha_1},\gen{x}{2}{\beta_1}\in G(x_1^\gamma x_2^\delta;x_1^{\alpha_2})$ which is impossible, since $Z$ has no chords. If $x_1^\gamma x_2^\delta\not\in M$ then $\gen{x}{1}{\alpha_2},\gen{x}{2}{\beta_2}\in G(x_1^{\alpha_1} x_2^\delta;x_1^\gamma)$ which is again impossible.

(6) The chordless cycle $Z$ has vertices $\gen{x}{1}{\alpha_2}$, $\gen{x}{2}{\beta_1}$, $\gen{x}{1}{\alpha_1}$, $\gen{x}{2}{\beta_2}$. 
Using the symmetries of the square, we can assume that $\alpha_1 < \alpha_2$ and $\beta_1 < \beta_2$. 
Since $\{\gen{x}{1}{\alpha_1},\gen{x}{1}{\alpha_2}\}, \{\gen{x}{2}{\beta_1},\gen{x}{2}{\beta_2}\}\in\min(\Bier_c(M))$, we have $m_1, m_2\in I_{SR}(\Bier_c(M))$ for $m_1:=x_{1,\alpha_1}x_{1,\alpha_2}$ and $m_2:=x_{2,\beta_1}x_{2,\beta_2}$.
Since $c_1\geq c_2\geq 2$, applying Theorem~\ref{FaceIdealMuraiThm}, we see that 
$m_1,m_2\in \pol(I_c(M))+\pol^*(I_c(M^\vee))$. It follows that either $\alpha_1=0, \alpha_2=1$ or $\alpha_1=c_1-1, \alpha_2=c_1$ and at the same time we have either $\beta_1=0, \beta_2=1$ or $\beta_1=c_2-1, \beta_2=c_2$. Using the symmetry between $M$ and $M^\vee$, we arrive at the next two possible cases.

a) $x_1^2\in I_c(M)$ and $x_2^2\in I_c(M^\vee)$. Due to Proposition~\ref{MinimalNonMonmomialsProp}, we have $(x_2^2)^c=x_1^{c_1}x_2^{c_2-2}\in\max(M)$, which divides $x_1^2\in I_c(M)$, a contradiction. 

b) $x_1^2, x_2^2\in I_c(M)$. By definition of the multicomplex ideal, this implies $M=\langle x_1x_2\rangle$. Hence, $I_c(M)=(x_1^2,x_2^2)$. Moreover, $M^\vee=\langle (x_1^2)^c,(x_2^2)^c\rangle=\langle x_1^{c_1-2}x_2^{c_2},x_{1}^{c_1}x_{2}^{c_2-2}\rangle$ and therefore $I_c(M^\vee)=((x_1x_2)^c)=(x_1^{c_1-1}x_2^{c_2-1})$. Thus, the Stanley--Reisner ideal of $\Bier_c(M)$ equals
$I_{SR}(\Bier_c(M)) = (x_{1,0}x_{1,1},x_{2,0}x_{2,1},x_{1,c_1}\ldots x_{1,2}x_{2,c_2}\ldots x_{2,2})$. This shows that $\Bier_c(M)=Z_4\ast\partial\Delta^{c_1+c_2-3}$. It is non-chordal and has an induced chordless 4-cycle on the vertex set $\{\gen{x}{1}{0},\gen{x}{1}{1},\gen{x}{2}{0},\gen{x}{2}{1}\}$. This finishes the proof.
\end{proof}

Thus, in the case $m=2$ we have the following classification of chordal Murai spheres $\Bier_c(M)$ with $c\in\N^m$.

\begin{theorem}
A Murai sphere $\Bier_{(c_1,c_2)}(M)$ is not chordal if and only if one of the following conditions holds:
\begin{enumerate}
\item $(c_1,c_2)=(2,1)$ and either $M=\langle x_1\rangle$, or $M^\vee=\langle x_1\rangle$;
\item $(c_1,c_2)=(2,1)$ and either $M=\langle x_1,x_2\rangle$, or $M^\vee=\langle x_1,x_2\rangle$;
\item $c_1\geq 2, c_2=1$ and either $M=\langle x_1x_2\rangle$, or $M^\vee=\langle x_1x_2\rangle$;
\item $c_1,c_2\geq 2$ and either $M=\langle x_1 x_2\rangle$, or $M^\vee=\langle x_1x_2\rangle$;
\item $c_1=2, c_2\geq 2$ and either $M=\langle x_1 x_2^{c_2-1}\rangle$, or $M^\vee=\langle x_1 x_2^{c_2-1}\rangle$;
\item $c_1=3, c_2\geq 2$ and $M=\langle x_1 x_2^{c_2}\rangle=M^\vee$.
\end{enumerate}
\end{theorem}
\begin{proof}
It suffices to observe that the statement is covered by Proposition~\ref{ChordalSpheresWithSecondCorrdinateOneProp} and Proposition~\ref{ChordalSpheresWithCoordinatesGreaterThanTwoProp}.    
\end{proof}

We finish this section with a classification of all the types of stacked Murai spheres.

\begin{example}\label{StackedWithC2EqualToOne}
Here is a list of all stacked spheres of the type $\Bier_c(M)$ with $c_1\geq 3$ and $c_2=1$. In what follows, we denote by $P_M$ the combinatorial simple polytope such that $\Bier_c(M)=\partial P^*_M$.
\begin{enumerate}
\item $M=\langle 1\rangle$. Then $I_c(M)=(x_1,x_2)$, $I_c(M^\vee)=(x_1^{c_1}x_2)$ and 
$I_{SR} = (x_{1,0},x_{2,0},x_{1,1}x_{1,2}\ldots x_{1,c_1}x_{2,1})$. Hence the sphere $\Bier_c(M)=\partial\Delta^{c_1}$ and $P_M = \Delta^{c_1}$.
\item $M=\langle x_1\rangle$. Then $I_c(M)=(x^2_1,x_2)$, $I_c(M^\vee)=(x_1^{c_1-1}x_2)$ and 
$I_{SR}=(x_{1,0}x_{1,1},x_{2,0},x_{1,2}x_{1,3}\ldots x_{1,c_1}x_{2,1})$. Hence the sphere $\Bier_c(M)=\partial\Delta^{c_1-1}\ast\partial\Delta^{1}$ and $P_M \in \vc^1(\Delta^{c_1})$.
\item $M=\langle x_2\rangle$. Then $I_c(M)=(x_1)$, $I_c(M^\vee)=(x_1^{c_1})$ and 
$I_{SR} = (x_{1,0},x_{1,1}x_{1,2}\ldots x_{1,c_1},x_{2,0}x_{2,1})$. Hence the sphere $\Bier_c(M)=\partial\Delta^{c_1-1}\ast\partial\Delta^{1}$ and $P_M \in \vc^1(\Delta^{c_1})$.
\item $M=\langle x_1, x_2\rangle$. Then $I_c(M)=(x^2_1,x_1x_2)$, $I_c(M^\vee)=(x_1^{c_1-1}x_2, x_1^{c_1})$ and 
$$
I_{SR}= (x_{1,0}x_{1,1},x_{1,0}x_{2,0}, x_{1,2}x_{1,3}\ldots x_{1,c_1}x_{2,1},x_{1,1}x_{1,2}\ldots x_{1,c_1},x_{2,0}x_{2,1}).
$$
Hence the sphere $P_M \in \vc^2(\Delta^{c_1})$.
\item $M=\langle x_1^{c_1-1}\rangle$. Then $I_c(M)=(x_1^{c_1},x_2)$, $I_c(M^\vee)=(x_1x_2)$ and $I_{SR}= (x_{1,0}x_{1,1}\ldots x_{1,c_1-1},x_{2,0},x_{1,c_1}x_{2,1})$. Hence the sphere $\Bier_c(M)=\partial\Delta^{c_1-1}\ast\partial\Delta^{1}$ and $P_M \in \vc^1(\Delta^{c_1})$.
\item $M=\langle x_1^{c_1-1},x_2\rangle$. Then $I_c(M)=(x_1^{c_1},x_1x_2)$, $I_c(M^\vee)=(x_1^{c_1},x_1x_2)$ and 
$$
I_{SR}= (x_{1,0}x_{1,1}\ldots x_{1,c_1-1},x_{1,0}x_{2,0},x_{1,1}x_{1,2}\ldots x_{1,c_1},x_{1,c_1}x_{2,1},x_{2,0}x_{2,1}).
$$
Hence the sphere $P_M \in \vc^2(\Delta^{c_1})$.
\item $M=\langle x_1^{c_1-1}, x_{1}^{c_1-2}x_2\rangle$. Then $I_c(M)=(x_1^{c_1},x_1^{c_1-1}x_2)$, $I_c(M^\vee)=(x_1^2,x_1x_2)$ and 
$$
I_{SR}= (x_{1,0}x_{1,1}\ldots x_{1,c_1-1},x_{1,0}x_{1,1}\ldots x_{1,c_1-2}x_{2,0},x_{1,c_1-1}x_{1,c_1},x_{1,c_1}x_{2,1},x_{2,0}x_{2,1}).
$$
Hence the sphere $P_M \in \vc^2(\Delta^{c_1})$.
\item $M=\langle x_1^{c_1}\rangle$. Then $I_c(M)=(x_2)$, $I_c(M^\vee)=(x_2)$ and
$I_{SR} = (x_{2,0},x_{2,1},x_{1,0}x_{1,1}\ldots x_{1,c_1})$. Hence the sphere $\Bier_c(M)=\partial\Delta^{c_1}$ and $P_M = \Delta^{c_1}$.
\item $M=\langle x_1^{c_1}, x_2\rangle$. Then $I_c(M)=(x_1x_2)$, $I_c(M^\vee)=(x_1^{c_1},x_2)$ and 
$I_{SR} = (x_{1,0}x_{2,0},x_{1,1}x_{1,2}\ldots x_{1,c_1},x_{2,1})$. Hence the sphere $\Bier_c(M)=\partial\Delta^{c_1-1}\ast\partial\Delta^{1}$ and $P_M \in \vc^1(\Delta^{c_1})$.
\item $M=\langle x_1^{c_1}, x_1^{c_1-2}x_2\rangle$. Then $I_c(M)=(x_1^{c_1-1}x_2)$, $I_c(M^\vee)=(x_1^{2},x_2)$ and 
$$
I_{SR} = (x_{1,0}x_{1,1}\ldots x_{1,c_1-2}x_{2,0}, x_{1,c_1-1}x_{1,c_1},x_{2,1}).
$$
Hence the sphere $\Bier_c(M)=\partial\Delta^{c_1-1}\ast\partial\Delta^{1}$ and $P_M \in \vc^1(\Delta^{c_1})$.
\item $M=\langle x_1^{c_1}, x_1^{c_1-1}x_2\rangle$. Then $I_c(M)=(x_1^{c_1}x_2)$, $I_c(M^\vee)=(x_1,x_2)$ and 
$$
I_{SR} = (x_{1,0}x_{1,1}\ldots x_{1,c_1-1}x_{2,0},x_{1,c_1},x_{2,1}).
$$
Hence the sphere $\Bier_c(M)=\partial\Delta^{c_1}$ and $P_M = \Delta^{c_1}$.
\end{enumerate}
\end{example}

\begin{proposition}
Let $c\in\N^2$ with $c_1\geq 3$ and $c_2=1$. Then $\Bier_c(M)$ is stacked if and only if it is one of the spheres described in Example~\ref{StackedWithC2EqualToOne}. In particular, in this case $P_M \in \vc^{k}(\Delta^{c_1})$ with $k=0,1,2$.
\end{proposition}
\begin{proof}
Due to~\cite[Theorem 8.5]{Kalai}, $\Bier_c(M)$ is stacked if and only if it is chordal and $p\in I_{SR}(\Bier_c(M))$ implies $\deg(p)=1,2$ or is greater than $|c|-2=c_1-1$. In view of Proposition~\ref{FaceIdealMuraiThm}, this is equivalent to 
$$
G(I_c(M)),G(I_c(M^\vee))\subseteq\{x_1,x_2,x_1x_2,x_1^2,x_1^{c_1},x_1^{c_1-1}x_2,x_1^{c_1}x_2\}.
$$

Application of Proposition~\ref{ChordalSpheresWithSecondCorrdinateOneProp} excludes the cases where either $M$ or $M^{\vee}$ is equal to $\langle x_1x_2\rangle$. Therefore, we have the list of stacked spheres enumerated in Example~\ref{StackedWithC2EqualToOne}. This finishes the proof.
\end{proof}

Our final result in this section describes all possible types of truncation polytopes corresponding to stacked Murai spheres.

\begin{theorem}
Suppose $c\in\N^m$ with $m\geq 2$ and $M$ is a proper $c$-multicomplex. \begin{itemize}
\item[(a)] For each $k$ with $1\leq k\leq m$, consider $M=\langle x_1,\ldots,x_k\rangle$. Then $\Bier_c(M)$ is a stacked sphere and we have:
$$
P_M \in \vc^k(\Delta^{|c|-1}).
$$
\item[(b)] If $\Bier_c(M)$ is a stacked sphere and $P_M \in \vc^{k}(\Delta^{|c|-1})$, then $0\leq k\leq m$.
\end{itemize}
\end{theorem}
\begin{proof}
To prove statement (a), note that the Stanley--Reisner ideal for $\Bier_c(M)$ has the form:
$$
I_{SR}=\pol(x_p^2,x_ix_j,x_q\,|\,1\leq p \leq k, k+1\leq q\leq m, 1\leq i < j\leq k) + \pol^*((x_p)^c\,|\,1\leq p\leq k).
$$
It is easy to see that $P_M$ is combinatorially equivalent to the simplex on the vertex set $\tilde{X}_1\cup\ldots\cup\tilde{X}_k$ to which $k$ pyramids are glued over the facets corresponding to the monomials of the form $(x_p)^c$ for $1\leq p\leq k$.

To prove statement (b), note that by~\cite[Theorem 8.5]{Kalai} the degrees of monomials in $G(I_{SR}(\Bier_c(M)))$ belong to the set $\{1,2,|c|-1,|c|\}$. Moreover, a minimal monomial generator of degree $|c|$ exists if and only if either $M$ or $M^\vee$ equals $\langle 1\rangle$. 

Each pyramid over a facet of the simplex $\Delta^{|c|-1}$ corresponds to an element of degree $|c|-1$ in $G(I_{SR}(\Bier_c(M)))$; that is, to an element of the form $(x_i)^c$ belonging to either $I_c(M)$ (that is the case if and only if $x_i\in\max(M^\vee)$) or $I_c(M^\vee)$ (that is the case if and only if $x_i\in\max(M)$), or to an element of the form $x_i^{c_i+1}$ whenever $c_i+1=|c|-1$. In the latter case, we may assume that $i=1$ and either $c=(c_1,2)$ or $c=(c_1,1,1)$.

Observe that the case $(x_i)^c\in I_c(M)\cap I_c(M^\vee)$ is equivalent to either $c_i=1$ and $\max(M)=\{x_i,x^{c-e_i-e_j}\,|\,1\leq j\neq i\leq m\}$, or $c_i\geq 2$ and $(x_i)^c=x_ix_j$ for some $1\leq j\neq i\leq m$. The latter is the case if and only if $c=(2,1)$ and $\max(M)=\{x_1,x_2\}$.

Let us consider each of the 4 particular cases we found above: 

\begin{itemize}
\item $c=(c_1,2), c_1\geq 2$ and $x_1^{c_1+1}\in G(I_{SR}(\Bier_c(M)))$.
It follows that $I_c(M), I_c(M^\vee)\subseteq\{x_1x_2, x_2, x_2^2\}$. Therefore, $G(I_{SR}(\Bier_c(M)))$ contains only one monomial of degree $|c|-1\geq 3$, hence $k\leq 1$.

\item $c=(c_1,1,1), c_1\geq 1$ and $x_1^{c_1+1}\in G(I_{SR}(\Bier_c(M)))$.
If $c_1=1,2$, then the result follows from the classification of all Murai spheres in dimensions 1 and 2. Assume that $c_1\geq 3$. It follows that $I_c(M), I_c(M^\vee)\subseteq\{x_1x_2, x_1x_3, x_2x_3, x_2, x_3\}$. Therefore, $G(I_{SR}(\Bier_c(M)))$ contains only one monomial of degree $|c|-1=c_1+1\geq 4$, hence $k\leq 1$.

\item $c=(2,1)$ and $\max(M)=\{x_1,x_2\}$.
Then $I_c(M)=(x_1^2, x_1x_2)$, $I_c(M^\vee)=(x_1x_2,x_1^2)$ and $P_M\in\vc^2(\Delta^2)$, hence $k=2$.

\item $c_i=1$ and $\max(M)=\{x_i,x^{c-e_i-e_j}\,|\,1\leq j\neq i\leq m\}$.
Then $I_c(M)=((x_i)^c,x_ix_j\,|\,1\leq i\neq j\leq m)$ and $I_c(M^\vee)=((x_i)^c,x_ix_j\,|\,1\leq j\neq i\leq m)$. The cases $m\leq 2$, $|c|=c_i+1$ for some $1\leq i\leq m$, and $|c|\leq 4$ have already been considered above, hence we can assume that $m\geq 3$, $|c|-1 > 3, c_i+1$ for all $1\leq i\leq m$. Therefore $G(I_{SR}(\Bier_c(M)))$ contains only two monomials of degree $|c|-1\geq 4$, hence $k\leq 2$. 
\end{itemize}

Finally, if none of the above cases holds, $k$ does not exceed $|\{x_1,x_2,\ldots,x_m\}|=m$. This finishes the proof.
\end{proof}

\section{Buchstaber and chromatic numbers}

In this section, we discuss the Buchstaber and chromatic numbers of Murai spheres. These combinatorial invariants are closely related to the topological properties of polyhedral products studied in the framework of toric topology~\cite{BP02, TT}.

\begin{definition}
Let $K$ be a simplicial complex on $[m]$ with $m$ real vertices. We call the {\emph{complex Buchstaber number}} of $K$ the maximal integer $r$ such that there exists a {\emph{characteristic map}} $\Lambda\colon [m]\to \Z^{m-r}$. We denote this number by $s(K)$. Each simplex of $K$ is mapped by $\Lambda$ to a part of a lattice basis of $\Z^{m-r}$. 

Similarly, given a prime number $p\in\N$, we call the {\emph{mod $p$ Buchstaber number}} of $K$ the maximal integer $r$ such that there exists a {\emph{mod $p$ characteristic map}} $\Lambda_p\colon [m]\to \Z_p^{m-r}$. We denote this number by $s_{p}(K)$. Each simplex of $K$ is mapped by $\Lambda_p$ to a linearly independent set of vectors in the $\Z_p$-vector space $\Z_p^{m-r}$.
\end{definition}

\begin{remark}
Observe that this definition is equivalent to~\cite[Definition 5.3]{BVV}. In the case $p=2$, it is equivalent to $s_\R(K)$ being equal to the maximal rank of a real toric subgroup in $\Z_2^{m}$ acting freely on the real moment-angle-complex $\rk$, see~\cite{Ayz10,E09,Er14,FM}.
\end{remark}

From the above definition, it follows that both $s(K)$ and $s_p(K)$ are combinatorial invariants of $K$ for any prime $p\in\N$.
Buchstaber invariant of a simplicial complex is known to be closely related to its chromatic number. Recall that given a simple graph $\Gamma=(V,E)$, its \emph{chromatic number} $\chi(\Gamma)$ equals the minimal number of colors needed for a proper coloring of $V$; that is, vertices of each edge in $E$ are colored in different colors. Now we define $\chi(K)$ to be equal to the chromatic number of the 1-skeleton $\sk^1(K)$ of a simplicial complex $K$.

Firstly, given a simplicial complex $K$ with $f_0(K)=m$ real vertices and the dimension $\dim(K)=n-1$, the basic inequality $s(K)\leq m-n$ is well-known; see~\cite{IZ01}. In the same paper, Buchstaber and chromatic numbers were linked by inequality $m-\chi(K)\leq s(K)$. Secondly, for any simplicial complex, the real Buchstaber number serves as an upper bound for the complex Buchstaber number: $s(K)\leq s_2(K)$, see~\cite[Fact 15, p.5]{E09} as well as~\cite{Ayz10, FM, Er14}. Finally, for every prime $p\neq 2$, this is proved by literally the same argument, see~\cite{BVV}, where the mod $p$ Buchstaber invariant was introduced and studied, and therefore one obtains the chain of inequalities:
$$
m-\chi(K)\leq s(K)\leq s_p(K)\leq m-n\text{ for any prime $p\in\N$}.
$$

It implies immediately that for any proper $c$-multicomplex $M$ the following inequalities hold:
$$
|c|-1\leq r_2(\Bier_c(M))\leq r(\Bier_c(M))\leq \chi(\Bier_c(M))\leq |\bar{c}|=|c|+m,
$$
where $r_p(K):=f_0(K)-s_p(K)$ and $r(K):=f_0(K)-s(K)$. Our next result shows that these estimates can be improved.

\begin{theorem}
Let $\Bier_c(M)$ be a Murai sphere. Then its Buchstaber numbers satisfy inequalities:
$$
f_0(\Bier_c(M))-|c|\leq s(\Bier_c(M))\leq s_p(\Bier_c(M))\leq f_0(\Bier_c(M))-|c|+1\text{ for any prime $p\in\N$}.
$$
\end{theorem}
\begin{proof}
Throughout the proof we denote by $v:=f_0(\Bier_c(M))$ and $n:=|c|-1=\dim(\Bier_c(M))+1$. According to the definition of Buchstaber invariants, it suffices to construct a map:
$\Lambda_{(p)}\colon [|\bar{c}|]\to\Z^{n+1}$, which sends each simplex to a part of a lattice basis, and then restrict this map to $[v]$. Denote the standard basis of the lattice $\Z^{|c|}$ by $\{e_{i,j}\,|\,1\leq i\leq m, 1\leq j\leq c_i\}$. Then we define:
$$
x_i^{j}\mapsto e_{i,j}\text{ for }1\leq i\leq m, 1\leq j\leq c_i\text{ and }x_i^{0}\mapsto \sum\limits_{k=1}^{c_i}e_{i,k}\text{ for }1\leq i\leq m.
$$

This shows that 
$v-n-1\leq s(\Bier_c(M))\leq s_p(\Bier_c(M))\leq v-n$,
which finishes the proof.
\end{proof}

Due to~\cite[Theorem 4.2]{LV}, each Bier sphere has the maximal possible Buchstaber number. We finish this section by showing that both possibilities for the value of Buchstaber invariant can occur for a Murai sphere.

\begin{theorem}\label{MainBuchThm}
Let $k\geq 3, c=(k+2,k)$, and $M=\langle x^a\mid |a|=k\rangle$ be a proper $c$-multicomplex. Then the Murai sphere $\Bier_c(M)$ is isomorphic to the boundary of the cyclic polytope $C(2k+4,2k+1)$ and $$
s(\Bier_c(M))=s_2(\Bier_c(M))=2.
$$
\end{theorem}
\begin{proof}
Recall that the Gale evenness condition provides a necessary and sufficient condition to determine a facet of a cyclic polytope. Let $T=\{t_1,\ldots,t_p\}$ with $t_1<t_2<\ldots<t_p$. Then a subset $T_{q}\subseteq T$ of cardinality $q$ forms a facet of a cyclic polytope (of the type) $C(p,q)$ if and only if any two elements in  $T\setminus T_{q}$ are separated by an even number of elements from $T_{q}$ in the sequence  $(t_{1},t_{2},\ldots ,t_{p})$. In the following, we make use of the notation $\Delta(p,q)$ for the $(q-1)$-dimensional simplicial sphere equal to the boundary of $C(p,q)$ and we denote by $V(K)$ the set of all real vertices of a simplicial complex $K$. 

The required isomorphism is given by the map $f\colon V(\Bier_c(M))\to T$: 
\begin{align*}
f(x_{1,i})&=t_{2i+2}, \text{ for }i=0,\ldots,k+1,\\
f(x_{2,i})&=t_{2(k-i)+3},\text{ for } i=0,\ldots,k,\\
f(x_{1,k+2})&=t_{1}.
\end{align*}
We have the following 3 possibilities for a facet. 

For $\sigma=V(\Bier_c(M))\setminus\{x_{1,a},x_{2,b},x_{2,c}\}$, where $a+b\le k$ and $a+c>k$, the element $f(x_{1,a})$ is in between the elements $f(x_{2,b})$ in $f(x_{2,c})$, so $f(\sigma)$ is a facet.

For $\sigma=V(\Bier_c(M))\setminus\{x_{1,a},x_{2,b},x_{1,c}\}$, where $a+b\le k$, $a+c>k$, and $c<k+2$, the element $f(x_{2,b})$ is in between the elements $f(x_{1,a})$ in $f(x_{1,c})$, so $f(\sigma)$ is a facet.

For $\sigma=V(\Bier_c(M))\setminus\{x_{1,a},x_{2,b},x_{1,c}\}$, where $a+b\le k$, and $c=k+2$, the element $f(x_{2,b})$ is to the right to both elements $f(x_{1,a})$ and $f(x_{1,c})=t_1$, so $f(\sigma)$ is a facet.

Therefore, the Gale evenness condition is satisfied in all the cases, and hence $f$ yields a simplicial isomorphism between $\Bier_c(M)$ and $\Delta(2k+4,2k+1)$. By~\cite[Theorem 1.4]{Hasui}, there exists neither a small cover, nor a quasitoric manifold over a dual cyclic polytope $C^*(n+3,n)$ for all $n\geq 6$. Hence for the Buchstaber numbers of $\Bier_c(M)$ we have:
$$
s(\Bier_c(M))\leq s_2(\Bier_c(M))<(2k+4)-(2k+1)=3.
$$
Since $s(\mathcal K)=1$ if and only if $\mathcal K$ is a simplex, we get $s(\Bier_c(M))=s_2(\Bier_c(M))=2$.
\end{proof}

The next three open problems arise naturally in this context. Recall that a $d$-sphere is called \emph{neighborly} if each set of its $\lceil\frac{d}{2}\rceil$ vertices forms a face. In particular, the boundaries of cyclic polytopes are neighborly spheres. 

\begin{problem}
Classify all neighborly Murai spheres.    
\end{problem}

As we have seen in Lemma~\ref{BierWithCompleteGraphLemma}, a Bier sphere is neighborly if and only if it is either the boundary of a simplex, or has dimension not greater than two. However, the last theorem shows that the class of Murai spheres contains neighborly spheres different from joins of boundaries of simplices, at least in each even dimension greater than $4$.  

\begin{problem}
Classify all Murai spheres with maximal possible Buchstaber numbers: 
$$
s(\Bier_c(M))=s_p(\Bier_c(M))=v-n\text{ for all prime }p\in\N,
$$
where $v=f_0(\Bier_c(M))$ and $n=\dim(\Bier_c(M))+1=|c|-1$.
\end{problem}

Due to~\cite[Theorem 4.2]{LV}, these equalities hold for all Bier spheres. On the other hand, Theorem~\ref{MainBuchThm} shows that this is not true for an arbitrary Murai sphere.

\begin{problem}
Find a non-polytopal Murai sphere.    
\end{problem}

Since the class of Bier spheres is a proper subclass in the class of Murai spheres, there exist infinitely many non-polytopal Murai spheres.

\subsection*{Acknowledgements}
The authors are grateful to Matvey Sergeev, Marinko Timotijevi\'c, and Rade \v{Z}ivaljevi\'c for numerous fruitful discussions, valuable comments and suggestions. Limonchenko was supported by the Serbian Ministry of Science, Technological Development and Innovation through the Mathematical Institute of the Serbian Academy of Sciences and Arts. Vavpeti\v{c} was supported by the Slovenian Research and Innovation Agency program P1-0292 and the grant J1-4031.

\normalsize

\end{document}